\documentclass{amsart}
\usepackage[utf8]{inputenc}

\usepackage{amssymb,amsmath,amsfonts,amsthm,mathrsfs,yhmath}
\usepackage{euler,palatino}
\usepackage{graphicx}
\usepackage{adjustbox}
\usepackage{mathrsfs}
\usepackage{mathtools}
\usepackage{enumitem}
\usepackage{tikz}
\usepackage[
top    = 1.00in,
bottom = 1.00in,
left   = 1.00in,
right  = 1.00in]{geometry}
\usepackage{tikz-cd}
\usepackage{ stmaryrd }
\usetikzlibrary{decorations.pathreplacing,decorations.markings}
\usetikzlibrary{arrows}

\tikzset{
commutative diagrams/.cd,
arrow style=tikz,
diagrams={>=latex}}

\numberwithin{equation}{section}

\newcounter{lecnum}

\newcommand{\lecture}[3]{
   \pagestyle{myheadings} \thispagestyle{plain}
   \newpage \setcounter{lecnum}{#1}
   \setcounter{page}{1}
   \noindent
   \begin{center}
   \framebox{
      \vbox{\vspace{2mm}
    \hbox to 6.58in { {\bf Math 738~Extremal Combinatorics
                        \hfill Carnegie Mellon University} }
    \hbox to 6.58in { {\bf Spring 2024} \hfill}
       \vspace{4mm}
       \hbox to 6.6in { {\Large \hfill Homework #1  \hfill} }
       \vspace{2mm}
       \hbox to 6.6in {  \hfill #2 }
      \vspace{2mm}}
   }
   \end{center}
   \markboth{Assignment #1: #2}{Assignment #1: #2}
   \vspace*{4mm}
}

\newcommand{\final}[2]{
   \pagestyle{myheadings} \thispagestyle{plain}
   \setcounter{page}{1}
   \noindent
   \begin{center}
   \framebox{
      \vbox{\vspace{2mm}
    \hbox to 6.58in { {\bf Math 603~Model Theory I
                        \hfill Carnegie Mellon University} }
    \hbox to 6.58in { {\bf Fall 2021} \hfill}
       \vspace{4mm}
       \hbox to 6.6in { {\Large \hfill Final  \hfill} }
       \vspace{2mm}
       \hbox to 6.6in {  \hfill #1 }
      \vspace{2mm}}
   }
   \end{center}
   \markboth{Final: #1}{Final: #1}
   \vspace*{4mm}
}

\newtheorem{theorem}{Theorem}[section]
\newtheorem{lemma}[theorem]{Lemma}
\newtheorem{proposition}[theorem]{Proposition}
\newtheorem{claim}[theorem]{Claim}
\newtheorem{corollary}[theorem]{Corollary}
\newtheorem{question}[theorem]{Question}
\newtheorem{fact}[theorem]{Fact}

\theoremstyle{definition}
\newtheorem{definition}[theorem]{Definition}
\newtheorem{remark}[theorem]{Remark}
\newtheorem{example}[theorem]{Example}

\newcommand{\cU}{\mathcal{U}}
\newcommand{\Fcal}{\mathcal{F}}

\newcommand{\QQ}{\mathbb Q}

\newcommand{\ZZ}{\mathbb Z}

\makeatletter
\newcommand*\bigcdot{\mathpalette\bigcdot@{.5}}
\newcommand*\bigcdot@[2]{\mathbin{\vcenter{\hbox{\scalebox{#2}{$\m@th#1\bullet$}}}}}
\makeatother

\begin{document}
\author[Bannister]{Nathaniel Bannister\\Carnegie Mellon University}

\title{Locales as spaces in outer models}
\begin{abstract}
    Let M be a transitive model of set theory and X be a space in the sense of M. Is there a reasonable way to interpret X as a space in V? A general theory due to Zapletal (\cite{Zap}) provides a natural candidate which behaves well on sufficiently complete spaces (for instance \v{C}ech complete spaces) but behaves poorly on more general spaces - for instance, the Zapletal interpretation does not commute with products. We extend Zapletal's framework to instead interpret locales, a generalization of topological spaces which focuses on the structure of open sets. Our extension has a number of desirable properties; for instance, localic products always interpret as spatial products. 
    We show that a number of localic notions coincide exactly with properties of their interpretations; for instance, we show a locale is $T_U$ if and only if all its interpretations are $T_1$, a locale is $I$-Hausdorff if and only if all its interpretations are $T_2$, a locale is regular if and only if all its interpretations are $T_3$, and a locale is compact if and only if all its interpretations are compact.
\end{abstract}

\maketitle

Given a model $M$ of ZFC and some $X$ a topological space of $M$, we would like to create some canonical interpretation of $X$ as a space with similar properties in some outer model of $M$. 
For sufficiently nice spaces (i.e. regular Hausdorff continuous open images of \v Cech-complete spaces), Zapletal defines a nice interpretation in \cite{Zap} and in the special case of Hausdorff spaces in \cite{Frem}; the main goal of this paper is to give a definition that holds for arbitrary spaces and locales and show that localic properties in the ground model transmit faithfully into properties of spaces in forcing extensions. 

Our interpretation factors through the context of frames and locales, of which we review the basic definitions and properties in section \ref{Frm_loc_section}. 
Intuitively, locales are an abstraction of topological spaces that focus purely on the open sets rather than points and include complete Boolean algebras, which generally have no points at all! 
Instead, locales are defined as the opposite category of algebraic objects called frames which behave like the open subsets of a space.
However, after some forcing a locale will have points, enough so that the structure of a locale transmits to the structure of a basis in the forcing extension and that distinct maps between locales interpret to distinct maps of spaces. 

Our extension from spaces to locales has a number of nice benefits. 
For instance, interpretation on the level of spaces does not preserve products; there are several nice examples of this in \cite[Section 7]{Zap} including both the rationals and the Sorgenfrey line. 
However, the interpretation of a product locale is always the product of spaces; in particular, failure for interpretation to commute with product of spaces can be traced exactly to the localic product being in general bigger than the spatial product. 
In particular, the interpretable spaces of \cite{Zap} are simply a nice class of spaces where several spatial notions coincide exactly with their localic counterparts. 
In combination with the examples of \cite[Section 7]{Zap}, we obtain interesting corollaries in the form of a plethora of examples of spatial locales whose products are not spatial. 
For instance, we obtain new forcing based proofs that for $X$ either $\mathbb{Q}$ or the Sorgenfrey line ${\Omega}(X)\oplus {\Omega}(X)$ is not spatial, the uncountable locale product $\bigoplus_{i<\omega_1}{\Omega}(\omega)$ is not spatial, and for $X$ the space of wellfounded trees on $\omega$, ${\Omega}(X)\oplus {\Omega}(\omega^\omega)$ is not spatial. 
The first of these examples appears as \cite[II.2.14]{Johnstone} but we have not been able to find mention of the other two examples. 

An interesting feature of our interpretation is that Hausdorff topological spaces \emph{do not} necessarily interpret to Hausdorff spaces in a forcing extension, nor even necessarily to $T_1$ spaces. 
Rather, the separation properties $T_1$ and $T_2$ holding in all forcing extensions correspond to localic strengthenings of the classical notions; namely, to the notions of $T_U$ and Isbell Hausdorffness respectively.
In contrast, the separation property $T_3$ translates directly to the analogous separation property in all forcing extensions. 

The structure of the paper is as follows. 
In section \ref{Frm_loc_section}, we review the basic properties of frames and locales. 
In section \ref{interp_section}, we define interpretation of locales as spaces in outer models and prove the existence and uniqueness of interpretations. 
In section \ref{force_section}, we analyze interpretation in the special case of interpreting a locale as a space in a forcing extension, where we gain a number of useful tools. 
Section \ref{glob_section} gives some basic properties that a space defined in all forcing extensions should satisfy and our definition of the interpretation of a locale does. 
We show that interpretation gives a full and faithful functor to a suitable category of such spaces.

The focus of section \ref{comp_section} is explicit examples; we compute the interpretation of complete metric spaces, compact Hausdorff spaces and, as a non-Hausdorff example, the spectrum of a commutative ring with $1$. 
In section \ref{sep_section}, we show that the separation properties $T_1,T_2$, and $T_3$ for interpretations in all forcing extensions correspond precisely with the notions of $T_U$, Isbell Hausdorfness, and regularity respectively. 
Though several of our proofs at first rely on forcing, some heavy handed use of Mostowski and Shoenfield absoluteness allows us to generalize the results from forcing extensions to all outer models, at least if we start with a proper class model of set theory. 
In section \ref{More_prop_section}, we note a few more properties of locales that are equivalent to properties of spaces in forcing extensions: for instance, a locale is compact if and only if all its interpretations are compact, a locale is $p$-connected if and only if it is connected in extensions that collapse a sufficiently large cardinal.
We show open monos of locales give rise to open embeddings, at least if the domain is sufficiently separated, and conclude with some basic analysis of localic groups, giving an interesting example of a generally nonspatial localic abelian group for every pair $A,B$ of discrete abelian groups which interprets to $\hom(A,B)$ with the compact open topology. 

Throughout, we assume general familiarity with set theoretic forcing; a good reference is \cite{Kunen}. 
We will often (sometimes without saying as much) invoke the equivalence between forcing with arbitrary posets and forcing with complete Boolean algebras. 
We assume less familiarity with locales and recall the relevant notions as they arise. 

\section{Frames and Locales} \label{Frm_loc_section}
We here recall the basic definitions of frames and locales that will be relevant to our later constructions. 
Some good references are \cite{FrmLoc} and \cite{Johnstone}.
\begin{definition}
    A \emph{frame} is a poset with all joins and all finite meets which satisfies the \emph{infinite distributive law}:
    \[a\wedge\left(\bigvee_{i\in I}b_i\right)=\bigvee_{i\in I}(a\wedge b_i).\]
    A frame map $f\colon L\to L'$ is an order preserving function that preserves all joins and all finite meets. 

    The category of locales is, by definition, the opposite of the category of frames. 
    That is, a locale is a frame and a map of locales $f\colon L\to L'$ is a frame map from $L'$ to $L$. 
\end{definition}
For a space $X$, the frame $\mathcal{O}X$ consists of all open subsets of $X$ ordered by containment; when written as a locale we use the notation $\Omega X$. 
A locale isomorphic to $\Omega X$ for some $X$ is called \emph{spatial}. 
If $f\colon X\to Y$ is a continuous function then $f$ induces a frame map $\mathcal{O}f\colon \mathcal{O}Y\to\mathcal{O}X$ - and consequently a locale map $\Omega f\colon\Omega X\to\Omega Y$ - given by taking the preimage of any open set.
However, there are more frames; for instance, any complete Boolean algebra (or more generally complete Heyting algebra) is a frame. 
Any frame can naturally be given the structure of a complete Heyting algebra; the Heyting implication is defined by $a\to b=\bigvee\{c\mid c\wedge a\leq b\}$ and arbitrary meets exist and can be expressed as $\bigwedge_{i\in I}a_i=\bigvee\{b\mid\forall i(b\leq a_i)\}$. 
However, frame maps do not generally preserve either the Heyting implication nor arbitrary meets. 
Maps that preserve both arbitrary meets and the Heyting implication are called \emph{open}; when $X$ and $Y$ are spaces and $Y$ is $T_1$, it is not too hard to show that $f\colon X\to Y$ is an open map if and only if $\Omega f\colon\Omega X\to \Omega Y$ is an open map of locales, though this can fail if $Y$ is not $T_1$. 

We note that if $\mathbb{A}$ and $\mathbb{B}$ are complete Boolean algebras and $\dot H$ is an $\mathbb{A}$ name for a $\mathbb{B}$ generic filter, then $\dot H$ induces a locale map $f_{\dot H}\colon\mathbb{A}\to\mathbb{B}$. 
Viewed as a frame map, $f_{\dot H}(b)=\llbracket b\in\dot H\rrbracket$; that is, $b$ maps to the largest condition forcing that $b$ is in $\dot H$. 
Moreover, this is all the locale maps from $\mathbb{A}$ to $\mathbb{B}$. 
Consequently, the category of locales has a piece that looks like topological spaces and continuous functions and another that looks like forcings and names for generic filters. 
Therefore, it should perhaps be little surprise that locales behave better than spaces when interacting with forcing. 

The terminal locale is denoted by $1$ and consists of exactly two points, which we denote as $0,1$. 
For a locale $L$, a \emph{point} of $L$ is a locale map $x\colon 1\to L$; viewed as a map of frames, a point of $L$ consists of a map $x\colon L\to \{0,1\}$ such that $x(1)=1$, $x(0)=0$, the preimage of $1$ is closed upwards and under finite meets, and whenever $x(\bigvee A)=1$ there is some $a\in A$ such that $x(a)=1$. 
A set of the form $x^{-1}(1)$ for some locale map $x\colon 1\to L$ is called a \emph{completely prime filter}. 
The space of completely prime filters $\Sigma L$ is called the \emph{spectrum} of $L$ and is topologized by declaring that for each $\ell\in L$, $\{\mathcal{F}\mid\ell\in\mathcal{F}\}$ is open. 
If $f\colon L\to L'$ is a locale map then $f$ induces a continuous function $\Sigma f\colon \Sigma L\to \Sigma L'$ obtained by taking the preimage of a completely prime filter under the frame map corresponding to $f$. 

The spectrum functor $\Sigma$ and the locale of open sets functor $\Omega$ are a pair of adjoint functors, with $\Sigma$ the left adjoint to $\Omega$. 
A space $X$ with $\Sigma\Omega X\cong X$ is called \emph{sober}. 
As a separation property, sobriety is strictly stronger than $T_0$, strictly weaker than $T_2$, and incomparable with $T_1$.
While information is often preserved in passing from spaces to locales, passage from locales to spaces tends to destroy a significant amount of information. 
For instance, the spectrum of a complete atomless Boolean algebra is empty. 
Indeed, a completely prime filter on a complete Boolean algebra would be a generic filter, which cannot exist in the ground model. 

We now define the relevant analogues of separation properties for locales. 
The analogue of the axiom $T_1$ that we will need is referred to as axiom $T_U$ in, for instance, \cite[III.1.5]{Johnstone} (short for totally unordered). 
Intuitively, the definition asserts that if $f$ and $g$ are continuous functions into $L$ such that $f(x)\in\overline{g(x)}$ for every point $x$ then $f=g$. 
\begin{definition}
    A locale $L$ is $T_U$ if for any locale maps $f,g\colon M\to L$, if $f\leq g$ pointwise when viewed as frame maps then $f=g$. 
\end{definition}

For our analysis of Hausdorffness and for connectedness, we will need the product of locales. 
\begin{definition} \label{prod_def}
    For $L,L'$ locales, we denote $L\oplus L'$ as the locale consisting of all downwards closed $A\subseteq L\times L'$ such that
    \begin{itemize}
        \item whenever $\{a_i\mid i\in I\}$ are in $L$ and $b\in L'$ such that $(a_i,b)\in A$ for each $i\in I$, we have $(\bigvee_ia_i,b)\in A$;
        \item whenever $a\in L$ and $\{b_i\mid i\in I\}$ are in $L'$ and $a\in L$ such that $(a,b_i)\in A$ for each $i\in I$, we have $(a,\bigvee_ib_i)\in A$. 
    \end{itemize}
    The locale map $\pi_1\colon L\oplus L'\to L$ is given on the level of frames as taking $\ell\in L$ to the smallest element of $L\oplus L'$ containing $(\ell,1)$. 
Similarly, we may define $\pi_2\colon L\oplus L'\to L$. 
\end{definition}
$L\oplus L'$ is the coproduct of $L$ and $L'$ in the category of frames (equivalently, the product in the category of locales); see \cite[Section IV.4]{FrmLoc}. 

We will later need the notions of a locale being \emph{Hausdorff} and being \emph{regular}. 
The most convenient definition of being Hausdorff for our purposes is the following. 
The idea is that a $T_0$ space $X$ is Hausdorff if and only if, letting $\Delta\subseteq X\times X$ be the diagonal, every open superset of $(X\times X)\setminus\overline{\Delta}$ is of the form $\bigcup\{U\times V\mid U\cap V\subseteq W\}$ for some open $W\subseteq X$. 

\begin{definition}
    A locale $L$ is {\em Isbell Hausdorff} or {\em $I$-Hausdorff} if whenever $U\in L\oplus L$ is such that $\{(\ell,\ell')\mid \ell\wedge \ell'=\varnothing\}\subseteq U$ there is a unique $\ell\in L$ such that $U=\{(a,b)\mid a\wedge b\leq \ell\}$. 
\end{definition}
The main difference between a space $X$ being Hausdorff and the locale $\Omega X$ being $I$-Hausdorff is that we use the localic product, which is generally larger than the spatial product- as a separation property, $\Omega X$ being $I$-Hausdorff is strictly between $X$ being Hausdorff and $X$ being $T_3$. 
An alternative useful characterization is given by \cite[Proposition 2.3]{FrmLoc}:
\begin{fact}
    $L$ is $I$-Hausdorff if and only if whenever $U\in L\oplus L$ is such that $\{(\ell,\ell')\mid \ell\wedge \ell'=\varnothing\}\subseteq U$ and $(a\wedge b,a\wedge b)\in U$, we have $(a,b)\in U$. 
\end{fact}
The idea for the localic analogue of $T_3$ is that a $T_0$ space is $T_3$ if and only if for each open set $U$ and point $x\in U$ there is an open $W$ with $x\in W\subseteq\overline{W}\subseteq U$. 
\begin{definition}
    For $\ell,\ell'\in L$, we write $\ell'\prec\ell$ and say $\ell'$ is well below $\ell$ if there is a $\ell''$ such that $\ell'\wedge\ell''=0$ and $\ell\vee\ell''=1$. 
A locale $L$ is \emph{regular} if for every $\ell\in L$, $\ell=\bigvee\{\ell'\mid \ell'\prec\ell\}$. 
\end{definition}
Unlike the axioms of $T_U$ and $I$-Hausdorff, regularity is a direct generalization of the spatial notion.

\section{Defining interpretation} \label{interp_section}
We give the following strengthening of the requirements of \cite[Definition 1.1]{Zap} for what an interpretation must satisfy; our strengthening makes uniqueness much easier to establish in general. 
\begin{definition}
    Suppose $M\models L$ is a frame. 
    A \emph{topological preinterpretation} of $L$ (over $M$) is a space $X$ in $V$ together with a map $\pi\colon L\to \mathcal{O}X$ such that 
    \begin{itemize}
        \item $\pi(1)=X$, $\pi(0)=\varnothing$;
        \item for all $\cU\in \mathcal{P}(L)\cap M$, $\pi(\bigvee\cU)=\bigcup_{\ell\in\cU}\pi(\ell)$;
        \item for all $\ell,\ell'\in L$, $\pi(\ell\wedge \ell')=\pi(\ell)\cap \pi(\ell')$. 
    \end{itemize}
    A preinterpretation $(X,\pi)$ of $L$ is the \emph{topological interpretation} of $L$ if for all topological preinterpretations $(Y,\rho)$ there exists a unique continuous $f\colon Y\to X$ such that the following diagram commutes: 
    \begin{center}
\begin{tikzcd}
\mathcal{O}X \arrow[rr, "\mathcal{O}f"] &                                        & \mathcal{O}Y \\
                                  & L \arrow[ru, "\rho"] \arrow[lu, "\pi"] &             
\end{tikzcd}
    \end{center}
\end{definition}
Note that the use of the article the is justified in defining the topological interpretation of $L$ since if $(X,\pi)$ and $(Y,\rho)$ are both interpretations of $L$ then the universal $f\colon X\to Y$ is necessarily a homeomorphism. 
Our first goal is to show that the topological interpretation of any frame always exists.

\begin{definition}
    Suppose $M\models L$ is a frame. 
    We define $\widehat{L}$ as the space with 
    \begin{itemize}
        \item points given by filters $\mathcal{F}$ on $L$ such that whenever $U\in\Fcal$ and $\cU\in M\cap\mathcal{P}(L)$ satisfies $\bigvee\cU=U$, $\Fcal\cap\cU\neq\varnothing$. 
        We say such a filter is \emph{completely prime} over $M$.
        \item topology generated by, for $\ell\in L$, declaring $\widehat{\ell}=\{\Fcal\mid \ell\in \Fcal\}$ to be open. 
    \end{itemize}
    If $f\colon L\to L'$ is a map of locales in $M$, we define $\widehat{f}\colon\widehat{L}\to\widehat{L'}$ by $\widehat{f}(\mathcal{F})=f^{-1}\mathcal{F}$. 
\end{definition}
We note that the construction of $\widehat{L}$ bears a striking resemblance to the construction of points of a frame except that we preserve only those joins that exist in $M$ - and these are the only joins that necessarily exist regardless. 
\begin{remark}
    We could dually define the points of $\widehat{L}$ as completely prime ideals over $M$; that is, ideals $I\subseteq L$ such that if $a\wedge b\in I$ then $a\in I$ or $b\in I$ and if $\cU\subseteq I$ and $\cU\in M$ then $\bigvee\cU\in I$. 
    We choose this presentation for the similarity with the generic filters of forcing. 
    Indeed, if $\mathbb{B}$ is a complete Boolean algebra in the sense of $M$ then $\widehat{\mathbb{B}}$ is the space of $M$-generic filters on $\mathbb{B}$ as a subspace of the Stone space of $\mathbb{B}$. 
\end{remark}

In contrast to the classical setting where many interesting locales have no points whatsoever, we note that $\widehat{L}$ will have a point in some forcing extension, for instance by forcing with $L$ itself. 
We will soon see that in some forcing extension, the structure of $L$ is interpreted faithfully as the structure of a basis of open sets and that distinct maps between locales are interpreted as distinct functions; see Proposition \ref{l_struct}. 

\begin{proposition} \label{locale_interp}
    $\widehat{L}$ with the function $\pi(\ell)=\widehat{\ell}$ is the topological interpretation of $L$. 
\end{proposition}
\begin{proof}
    Suppose $(Y,\rho)$ is a preinterpretation of $L$. 
    Let $f\colon Y\to\widehat{L}$ be defined by $f(x)=\{\ell\in L\mid x\in \rho(\ell)\}$. 
    Note that $f(x)$ is indeed a completely prime filter and that $f^{-1}\widehat{\ell}=\rho(\ell)$ so that $f$ is continuous as the preimage of every basic open set is open and $\rho=f^{-1}\circ\pi$. 
    Moreover, if $g\colon Y\to\widehat{L}$ is any continuous function such that $g^{-1}\circ\pi=\rho$ then for each $x\in Y$, 
    \[\begin{aligned}
        g(x)&=\{\ell\mid g(x)\in\widehat{\ell}\}\\
        &=\{\ell\mid x\in g^{-1}\widehat{\ell}\}\\
        &=\{\ell\mid x\in f^{-1}\widehat{\ell}\}\\
        &=f(x)
    \end{aligned}\]
    so $f=g$. 
\end{proof}
\begin{remark} \label{Borel}
    The definition of a completely prime filter over $M$ requires only quantification over elements of $L$ and $\mathcal{P}(L)\cap M$. 
    In particular, in any universe where $\mathcal{P}(L)\cap M$ is countable, the space of completely prime filters over $L$ is a Borel set in $\mathcal{P}(L)$ of fairly low complexity. 
    We will leverage this fact in later sections to generalize results in forcing extensions to results about arbitrary $M$.
\end{remark}

Note that we also obtain an interpretation of a space as in the following definition, which is a slight stengthening of \cite[Definition 1.1]{Zap}:
\begin{definition}
    Suppose $M\models X$ is a topological space. 
    A \emph{topological preinterpretation} of $X$ (over $M$) consists of a topological space $Y$ with a pair of functions $\pi_X\colon X\to Y$ and $\pi_{\mathcal{O}X}\colon \mathcal{O}X\to\mathcal{O}Y$ such that 
    \begin{itemize}
        \item $\pi_{\mathcal{O}X}(\varnothing)=\varnothing$ and $\pi_{\mathcal{O}X}(X)=Y$;
        \item for every $x\in X$ and open set $O\in\mathcal{O}X$, $x\in O$ if and only if $\pi_X(x)\in \pi_{\mathcal{O}X}(O)$;
        \item for every $\cU\in M$ with $\cU\subseteq\mathcal{O}X$, $\pi_{\mathcal{O}X}(\bigcup\cU)=\bigcup_{W\in\cU}\pi_{\mathcal{O}X}(W)$ and if $U,W\in\mathcal{O}X$ then $\pi_{\mathcal{O}X}(U\cap W)=\pi_{\mathcal{O}X}(U)\cap\pi_{\mathcal{O}X}(W)$.
    \end{itemize}
    A topological preinterpretation $(Y,\pi_X,\pi_{\mathcal{O}X})$ of $X$ is the \emph{topological interpretation} of $X$ if for any topological preinterpretation $(Z,\rho_X,\rho_{\mathcal{O}X})$ of $X$ there is a unique continuous $f\colon Z\to Y$ such that $\rho_X=f\circ \pi_X$ and $f^{-1}\pi_{\mathcal{O}X}(U)=\rho_{\mathcal{O}X}(U)$ for any $U\in\mathcal{O}X$. 
\end{definition}
Once again, with this definition the topological interpretation is unique up to a unique structure preserving homeomorphism. 
The point is the analogue of Proposition \ref{locale_interp} also holds for topological interpretations of spaces. 
\begin{proposition}
    For any space $X$, $\widehat{\mathcal{O}X}$ with $\pi_X$ mapping $x\in X$ to the principal filter at $x$ and $\pi_{\mathcal{O}X}(U)=\widehat{U}$ is the topological interpretation of $X$. 
\end{proposition}
\begin{proof}
    Suppose that $(Z,\rho_X,\rho_{\mathcal{O}X})$ is any topological preinterpretation of $X$. 
    Note that $(Z,\rho_{\mathcal{O}X}))$ is a topological interpretation of $\mathcal{O}X$ so by Proposition \ref{locale_interp}, there is a unique continuous $f\colon Z\to Y$ such that $\rho_{\mathcal{O}X}=f^{-1}\circ\pi_{\mathcal{O}X}$. In particular, all we must show is that $f\circ\rho_X=\pi_X$. 
    Indeed for each $x\in X$,
    \[\begin{aligned}
        f\circ\rho_X(x)&=\{U\in\mathcal{O}X\mid f\circ\rho_X(x)\in \widehat{U}\}\}\\
        &=\{U\in\mathcal{O}X\mid\rho_X(x)\in f^{-1}\widehat{U}\}\\
        &=\{U\in\mathcal{O}X\mid \rho_x(x)\in\rho_{\mathcal{O}X}(U)\}\\
        &=\{U\mid\mathcal{O}X\mid x\in U\}
    \end{aligned}\]
    with the first equality following from the definition of $\widehat{L}$, the second a standard manipulation, the third using that $\rho_{\mathcal{O}X}=f^{-1}\circ\pi_{\mathcal{O}X}$, and the fourth using the second defining property of topological preinterpretations. 
\end{proof}
\begin{remark}
    In the special case where $M=V$, $\widehat{\mathcal{O}X}$ is the soberification of $X$. 
    In particular, if we do not take any outer model, a space interprets to itself if and only if it is sober; this class includes most spaces of interest including all Hausdorff spaces. 
\end{remark}
Despite being nicely defined for spaces, we will see interpretation abstractly behaves much nicer with locales. 
For instance, we will see in the next section that interpretation commutes with arbitrary products (and more generally limits) of locales but interpretation \emph{does not} commute with even binary products of topological spaces. 

\section{Interpretation into forcing extensions} \label{force_section}
A special case of particular interest occurs when passing to a forcing extension. 
Throughout, $V$ will play the role of $M$ from the previous section and $V[G]$ will play the role of $V$. 
We will always The following proposition is a basic tool we will reference throughout our analysis. 
\begin{proposition} \label{pt-map}
    Suppose $\mathbb{B}$ is a complete Boolean algebra in $V$ and $L$ is a locale in $V$. There is a bijective correspondence between $\mathbb{B}$ names for points of $\widehat{L}$ up to forced equality and locale maps from $\mathbb{B}$ to $L$. 
    In particular, the points of $\widehat{L}$ in $V[G]$ are in bijective correspondence with the colimit in sets of the ground model hom sets $\operatorname{colim}_{b\in G}\operatorname{hom}_{\operatorname{Loc}}(\downarrow b,L)$.
\end{proposition}
\begin{proof}
    Given a name $\dot\Fcal$ for an element of $\widehat{L}$, let $f_{\dot\Fcal}\colon L\to \mathbb{B}$ be $f_{\dot\Fcal}(l)=\llbracket\check\ell\in\dot\Fcal\rrbracket$; 
    There is some tedious but straightforward work to show that $f_{\dot\Fcal}(l)$ is indeed a frame map. 
    For a frame map $f\colon L\to \mathbb{B}$, let $\dot\Fcal_f$ be a name for $\{\ell\mid f(\ell)\in\dot G\}$. 
    We claim these two functions are inverses of each other. 

    Indeed, given a name for a point $\dot{\mathcal{H}}$ of $\widehat{L}$, $b\in\mathbb{B}$, and $\ell\in L$, 
    \[\begin{aligned}
        b\Vdash \ell\in\dot{\mathcal{H}} &\Leftrightarrow b\leq \llbracket\ell\in \dot{\mathcal{H}}\rrbracket\\
        &\Leftrightarrow b\Vdash \llbracket\ell\in \dot{\mathcal{H}}\rrbracket\in\dot G\\
        &\Leftrightarrow b\Vdash \ell\in \dot \Fcal_{f_{\dot{\mathcal{H}}}}.
    \end{aligned}.\]
    In particular, $\Vdash\dot{\mathcal{H}}=\dot\Fcal_{\dot {f_{\dot{\mathcal{H}}}}}$.

    Conversely, if $g\colon L\to\mathbb{B}$ is a frame map and $\ell\in L$,
    \[\begin{aligned}
        g(\ell)&=\llbracket g(\ell)\in\dot G\rrbracket\\
        &=\llbracket \ell\in\dot\Fcal_g\rrbracket\\
        &=f_{\dot\Fcal_g}(\ell).
    \end{aligned}\]
\end{proof}

We note here that the structure of locales and maps between them transfers faithfully to the structure of a basis for open sets and continuous functions respectively after a sufficiently large collapse extension. 
We note that in each equivalence, the locale version implies the space version in all forcing extensions; we only need the full collapse to recover the full structure of the locales. 
\begin{proposition} \label{l_struct}
    Suppose $L,L'$ are locales and $\lambda\geq |L|,|L'|$. After forcing with $\operatorname{Coll}(\omega,\lambda)$, 
    \begin{enumerate}
        \item whenever $\ell\in L$ and $\ell\neq0$, $\widehat{\ell}$ is nonempty; \label{nonz}
        \item for all $\ell,\ell'\in L$, $\ell\leq\ell'$ if and only if $\widehat{\ell}\subseteq\widehat{\ell'}$; \label{sym}
        \item for each $\ell\in L$, $\widehat{\ell\rightarrow0}$ is the interior of the complement of $\widehat{\ell}$ for $\ell\rightarrow0$ the Heyting negation $\bigvee\{\ell'\mid \ell\wedge\ell'=0\}$; \label{comp}
        \item if $f,g\colon L\to L'$ are frame maps with $f\neq g$ then $\widehat{f}\neq\widehat{g}$; \label{neq} 
        \item if $f\colon L\to L'$ is a frame map which is not injective then $\widehat{f}$ is not surjective; \label{quo} 
        \item a frame map $f\colon L\to L'$ is surjective if and only if $\widehat{f}$ is an embedding; \label{emb}
        \item a frame map $f\colon L\to L'$ is an isomorphism if and only if $\widehat{f}$ is a homeomorphism. \label{homeo}
        
    \end{enumerate}
\end{proposition}
\begin{proof}
    For (\ref{nonz}), note that any $V$-generic filter on $L$ is a point of $L$ and $\operatorname{Coll}(\omega,\lambda)$ adds $V$-generic filters for all posets of size at most $\lambda$. 
    
    For (\ref{sym}), if $\ell\not\leq\ell'$ then forcing with $\{a\in L\mid a\leq \ell, a\not\leq\ell'\}$ adds a generic point to $\widehat{\ell}$ which is not in $\widehat{\ell'}$. 

    For (\ref{comp}), $\widehat{\ell\rightarrow0}$ is an open set disjoint from $\widehat{\ell}$ and if $b$ is any basic open set disjoint from $\ell$ then $b\wedge\ell=0$ by $(1)$ so $b\leq \ell\rightarrow0$, as desired. 

    For (\ref{neq}), if $f\neq g$ then we may fix some $\ell\in L$ with $f(\ell)\neq g(\ell)$. 
    By (\ref{sym}), there is an $x\in \widehat{f(\ell)}\triangle \widehat{g(\ell)}$. 
    Then $\ell\in \widehat{f}(x)\triangle \widehat{g}(x)$; in particular, $\widehat{f}(x)\neq\widehat{g}(x)$. 

    For (\ref{quo}), suppose $f$ is not injective and $\ell\neq\ell'$ are in $L$ with $f(\ell)= f(\ell')$. 
    Then $\widehat{f}^{-1}\widehat{\ell}=\widehat{f}^{-1}\widehat{\ell'}$. 
    Then any $x\in\widehat{\ell}\triangle\widehat{\ell'}$ is not in the range of $\widehat{f}$; note such an $x$ exists by (\ref{sym}). 

    For (\ref{emb}), first suppose $f$ is surjective. 
    First note that $\widehat{f}$ is injective as for any $x,y\in\widehat{L'}$, if $\ell\in x\triangle y$ and $f(\ell^*)=\ell$ then $\ell^*\in \widehat{f}(x)\triangle \widehat{f}(y)$. 
    Moreover, for any $\ell\in\widehat{L'}$, if $f(\ell^*)=\ell$ then $\widehat{f}[\widehat{\ell}]=\widehat{f}[\widehat{L}']\cap\widehat{\ell^*}$ is relatively open. 
    Conversely, if $f$ is not surjective we may fix some $\ell\in L'$ not in the range of $f$. 
    Let $a^*=\bigvee\{a\mid a\in L, f(a)\leq\ell\}$; note that $f(a^*)<\ell$ since $f$ preserves arbitrary joins and $\ell$ is not in the image of $f$. 
    Then if $d\in L$ is such that $\widehat{f}^{-1}\widehat{d}\subseteq\widehat{\ell}$ then $\widehat{f}^{-1}\widehat{d}=\widehat{f(d)}\subset \widehat{f(a^*)}$. 
    In particular, if $x\in\widehat{\ell}\setminus\widehat{\ell^*}$ then $f(x)$ is not in the relative interior of $\widehat{f}[\widehat{\ell}]$ (=$\widehat{a^*}\cap\widehat{f}[L']$). 

    Finally for (\ref{homeo}), if $\widehat{f}$ is a homeomorphism then by (\ref{quo}) and (\ref{emb}) $f$ is bijective and therefore an isomorphism. 
    Conversely, if $f$ is an isomorphism then $\widehat{f}$ is a homeomorphism with inverse $\widehat{f^{-1}}$.

\end{proof}

    \begin{remark}
        The converse of item (\ref{quo}) in Proposition \ref{l_struct} is false. 
        For instance, let $X$ be $\mathbb{R}$ with the discrete topology, $Y$ be $\mathbb{R}$ with the euclidean topology, and $f\colon \mathcal{O}Y\to\mathcal{O}X$ the natural inclusion map. 
        Then $f$ is injective but $\widehat{\mathcal{O}X}$ is the ground model reals with the discrete topology (as $\{\{x\}\mid x\in\mathbb{R}\}$ forms a cover of $X$) but, as we will see in Proposition \ref{complete_met}, $\widehat{\mathcal{O}Y}$ is the reals of the forcing extension. 
    \end{remark}

    \begin{remark}
        The monomorphisms of locales (i.e. the epimorphisms of frames) are generally viewed as a somewhat wild class of maps, see e.g. \cite[Chapter IV.6]{FrmLoc}. 
        However, a consequence of Proposition \ref{adj} is that monic arrows of locales always interpret to injective maps of spaces and by item \ref{neq} of Proposition \ref{l_struct}, maps which always interpret to an injection of spaces must be monic. 
        In particular, the difference between monic maps of locales and surjective maps of frames is detecting precisely the difference between injective continuous maps and embeddings. 
    \end{remark}

    Proposition \ref{pt-map} is a special case of an adjunction between locales in the ground model and names for spaces in a generic extension that we now construct. 
    The main corollary is that interpretation maps arbitrary limits, in particular products, of locales to the corresponding limit of spaces in a generic extension. 
    We contrast this with \cite[Theorem 3A]{Frem}, where products of spaces interpreted correctly only when all but countably many of the spaces are compact. 
    However, our results do come with the caveat that the product of locales may be quite a bit bigger than the product of spaces. 

    We first define a suitable category for interpretation to map into. 

    \begin{definition}
         For $\mathbb{B}$ a complete Boolean algebra, we denote by $\operatorname{Top}^{\mathbb{B}}$ the category where
    \begin{itemize}
        \item Objects are $\mathbb{B}$-names for topological spaces. 
        \item A morphism from $\dot X$ to $\dot Y$ is an equivalence class of names $\dot f$ for a continuous function from $\dot X$ to $\dot Y$, where we identify $\dot f$ and $\dot g$ if and only if $\Vdash_{\mathbb{B}}\dot f=\dot g$.
    \end{itemize}
    \end{definition}
   
    Note that $L\mapsto\widehat{L}$ gives a functor from $\operatorname{Loc}$ to $\operatorname{Top}^{\mathbb{B}}$ and if $\mathbb{B}$ is the trivial forcing then $\operatorname{Top}^{\mathbb{B}}$ is the normal category of topological spaces and continuous maps. The main goal of this section is to construct a left adjoint to $\,\widehat{}\,$; the main corollary is that interpretation preserves arbitrary limits of locales. 

    \begin{definition} \label{namespace}
        Given $\dot X\in \operatorname{Top}^{\mathbb{B}}$, the \emph{namespace locale} for $X$, denoted $L_{\dot X}$ is the locale with underlying set given by $\{U\mid\Vdash_{\mathbb{B}}U\subseteq\dot X\text{ open}\}/\simeq$ where $U\simeq W$ if and only if $\Vdash_{\mathbb{B}}U=W$. 
        Then declaring $U\leq W$ if and only if $\Vdash U\subseteq W$ turns $L_{\dot X}$ into a locale, with the join $\bigvee\cU$ a name for $\bigcup\cU$ and the meet $U\wedge W$ a name for $U\cap W$.
    \end{definition}

        An interesting special case is when $\dot X$ is a name for a one point space $*$. 
        Then $L_{\dot X}$ consists of all names $\dot U$ which are forced to be either $\varnothing$ or $*$; $U\mapsto\llbracket U=*\rrbracket$ then gives an isomorphism between $L_{*}$ and the complete Boolean algebra $\mathbb{B}$ itself.

        Note that $L_{\dot X}$ is naturally a functor from $\operatorname{Top}^{\mathbb{B}}$ to $\operatorname{Loc}$: if $\dot f\colon\dot X\to\dot Y$ is a name for a continuous function then we set $L_{\dot f}$ to be the frame map which assigns $U\in L_{\dot Y}$ to a name for the preimage of $U$ under $\dot f$. 
    
    The following is a generalization of the usual adjunction between spaces and locales. 
    \begin{proposition} \label{adj}
        For a locale $L$ and a name for space $\dot X$, there is a natural bijection between 
        \begin{itemize}
            \item Names for continuous functions $\dot f\colon\dot X\to \widehat{L}$ up to forced equality and
            \item Frame maps from $L$ to $L_{\dot X}$. 
        \end{itemize}
        
        That is to say, $L_{-}$ is the left adjoint to $\,\hat{}\,\colon \operatorname{Loc}\to \operatorname{Top}^{\mathbb{B}}$.
    \end{proposition}
    \begin{proof}
        Given a frame map $f\colon L\to L_{\dot X}$, we set $\varphi_f\colon \dot X\to\widehat{L}$ to be a name for the function which maps each $x\in X$ to $\{\ell\in L\mid x\in f(\ell)(\dot G)\}$. 
        First, note that $\varphi_f(x)\in\widehat{L}$ for each $x\in X$: $\varphi_f(x)$ is a filter since $f$ is order preserving and maps meets to names for intersections and if $\ell\in \varphi_f(x)$ and $\bigvee\cU=\ell$ then $\Vdash\bigcup_{a\in\cU}f(a)=f(\ell)$ so that $\varphi_f(x)\cap\cU\neq\varnothing$. 
        Moreover, $\varphi_f$ is continuous as the preimage of any basic open $\widehat{\ell}$ is the open set $f(\ell)(\dot G)$. 

        The inverse  of $f\mapsto \varphi_f$ assigns a name for a continuous function $\dot f\colon X\to\widehat{L}$ to the frame map $\psi_{\dot f}$ that associates $\ell\in L$ to a name for $\dot f^{-1}\widehat{\ell}$. 
        Indeed, if $f\colon L\to L_{\dot X}$ is a map of frames then $\psi_{\varphi_f}(\ell)$ is a name for $\varphi_f^{-1}\ell=f(\ell)$. 
        Conversely, if $\dot f$ is a name for a continuous function then for each $x\in X$,
        \[\begin{aligned}
            \varphi_{\psi_{\dot f}}(x)&=\{\ell\in L\mid x\in \psi_{\dot f}(\ell)(\dot G)\}\\
            &=\{\ell\in L\mid x\in f^{-1}\widehat{\ell}\}\\
            &= f(x).
        \end{aligned}\]
    \end{proof}
    We note that that the adjoint pair \,$\widehat{}$\, and $L_{-}$ are, in general, not idempotent, so there is little hope that we may restrict to an equivalence between a nice category of locales and a nice category of names for spaces as we can in the classical setting between sober spaces and spatial locales. 
    For instance, suppose $\dot X$ is a name for the one point space. Then $L_{-}$ is the locale of names which are forced to be either empty or all of $\dot X$, which is isomorphic to the complete Boolean algebra $\mathbb{B}$ itself. In particular, $\widehat{L}_{\dot X}$ is the space of $V$-generic filters on $\mathbb{B}$, which is very rarely a singleton. 

    \begin{remark}
        We could similarly define $\operatorname{Loc}^{\mathbb{B}}$ as locales and names for locale maps up to forced equality. 
        Remarkably, $\operatorname{Loc}^{\mathbb{B}}$ is equivalent to the slice category of locales over $\mathbb{B}$ (see \cite[Theorem 1.6.3]{elephant}). 
        Our interpretation can be viewed as starting with a locale $L$ and mapping it to the name for the locale corresponding to the projection map $L\oplus\mathbb{B}\to\mathbb{B}$ and then applying the spectrum functor in the forcing extension.
        However, we obtain interesting results mapping from locales to spaces rather than locales to locales coming from the additional input of points. 
        For instance, we obtain a test for whether some family is a cover in a locale by checking whether all of its interpretations are covers; this will give us a new proof that certain products of spatial locales are not spatial. 
        Moreover, we obtain a reasonable way to compare locale notions with spatial notions in forcing extensions. 
        While there is some information loss in passing to individual forcing extensions, we will see in Section \ref{glob_section} that mapping locales to a suitable notion of spaces defined in all forcing extensions is full and faithful. 
    \end{remark}

    Note that limits in $\operatorname{Top}^{\mathbb{B}}$ are just names for the limit of spaces. 
    Since right adjoints preserve all limits, a nice corollary of Proposition \ref{adj} is the following:

    \begin{corollary} \label{prod_preserve}
        Interpretation of an arbitrary limit of locales is the corresponding limit of spaces. 
        In particular, the interpretation of a product of locales is the product of their interpretations. 
    \end{corollary}

    \begin{remark}
        Interpretation also preserves arbitrary coproducts of locales. 
        Indeed, if $\langle L_i\mid i\in I\rangle$ are locales, their coproduct is given by the pointwise product $\prod_iL_i$ with pointwise join and pointwise meet. 
        Then any $\ell\in\prod_iL_i$ is the join of conditions below $\ell$ which have all but one coordinate equaling $0$. 
        Interpretation does not, however, preserve arbitrary colimits; see Proposition \ref{union_prop} for a special case where colimits are preserved.
    \end{remark}

    We obtain the following test for the product of spatial locales to be spatial that will allow us to transfer both the positive and negative results of \cite{Zap} to consequences for the corresponding locales; first we cite \cite[Lemma II.2.13]{Johnstone}
    \begin{lemma} \label{II.2.13}
        Let $\langle X_i\mid i\in I\rangle$ be a family of spaces. Then the natural map from $\Omega(\prod_{i\in I}X_i)$ to $\bigoplus_{i\in I}\Omega(X_i)$ is an isomorphism of locales if and only if $\bigoplus_{i\in I}\Omega(X_i)$ is spatial.
    \end{lemma}
    Our analysis allows us to add a third clause:
    \begin{corollary} \label{prod_cor}
        Suppose $\langle X_i\mid i\in I\rangle$ are spaces. 
        The following are equivalent:
        \begin{enumerate}
            \item The natural locale map 
            \[{\Omega}\left(\prod_{i\in I}X_i\right)\to \bigoplus_{i\in I}{\Omega}X_i\]
            is an isomorphism; \label{iso_frame}
            \item $\bigoplus_{i\in I}{\Omega}X_i$ is spatial; \label{spatial}
            \item in all forcing extensions, the natural map \[\widehat{\Omega\left(\prod_{i\in I}X_i\right)}\to \prod_{i\in I}\widehat{\Omega X_i}\]
            is a homeomorphism. \label{prod_all}
        \end{enumerate}
        \begin{proof}
            $(\ref{iso_frame}\Leftrightarrow\ref{spatial})$ is Lemma \ref{II.2.13}. 
            $(\ref{iso_frame})\Rightarrow(\ref{prod_all})$ follows since interpretation preserves arbitrary products of locales by Proposition \ref{adj}. 
            For $(\ref{prod_all})\Rightarrow(\ref{iso_frame})$, if the locale map of (\ref{iso_frame}) is not an isomorphism then by item \ref{homeo} of Proposition \ref{l_struct}, this map does not interpret to a homeomorphism of spaces in all sufficiently large collapse extensions. 
            Since interpretation does preserve arbitrary products of locales by Proposition \ref{adj}, the map in (\ref{prod_all}) is not a homeomorphism in all sufficiently large collapse extensions. 
            
        \end{proof}
    \end{corollary}

    Corollary \ref{prod_preserve} can fail horribly for products of spaces, as examples 7.8, 7.9, 7.10, and 13.3 of \cite{Zap} show. 
    In particular, the products of spaces in each of these examples must be different to the localic product. 
    As a consequence of Corollary \ref{prod_cor}, we obtain the following corollary applied to each of these examples. 
    \begin{corollary} \label{non_spatial}
        
        Suppose $X$ is either $\mathbb{Q}$ with the topology inherited as a subspace of $\mathbb{R}$ or the Sorgenfrey line. 
        Then ${\Omega}X\oplus {\Omega}X$ is not spatial. 
        If $X$ is the space of wellfounded trees on $\omega$ then ${\Omega}X\oplus {\Omega}(\omega^\omega)$ is not spatial. 
        The uncountable localic product $\bigoplus_{i<\omega_1}{\Omega}(\omega)$ is not spatial. 
    
    \end{corollary}
    Note that for $X=\mathbb{Q}$ the result already appears as \cite[Lemma II.2.14]{Johnstone}; one nice point about our proof is that we need no explicit computations involving the localic product. 
    \begin{remark}
        \cite[Claim 13.3]{Zap} is an example of a space where interpreting a space in an intermediate extension and then in the full extension yields a different result than interpreting directly to the full extension, namely the product of $\omega_1$ many copies of the natural numbers $\omega^{\omega_1}$. 
        This problem does not quite type check in the present context: our interpretation goes from locales to spaces rather than spaces to spaces. 
        We still, however, obtain interesting information from this example and, while the full strength of \cite[Claim 13.3]{Zap}
        requires the continuum hypothesis, the relevant direction of the proof for Corollary \ref{non_spatial} does not and we may still obtain that $\bigoplus_{i<\omega_1}\Omega(\omega)$ is nonspatial. 
        Indeed, if $T\subseteq\omega^{<\omega_1}$ is any Aronszajn tree, fix enumerations $\{t_{\alpha}^n\mid n<\omega\}$ of each level of $T$. 
        The family $U_{\alpha\beta}=\{x\in\omega^{\omega_1}\mid t_\alpha^{x(\alpha)}\not\leq t_\beta^{x(\beta)}\}$ is a cover of $\omega^{\omega_1}$ which is no longer a cover in any forcing with a branch through $T$ (for instance after collapsing $\omega_1$).
    \end{remark}
    The example of $\mathbb{Q}^2$ generalizes greatly to produce a plethora of spatial locales whose product is not spatial. 
    \begin{proposition}
        Suppose $X$ is a nonempty Hausdorff space and $A,B\subseteq X$ are disjoint and dense. 
        Then ${\Omega}(A)\oplus{\Omega}(B)$ is not spatial. 
    \end{proposition}
    \begin{proof}
        By Corollary \ref{prod_cor}, it suffices to show that ${\Omega}(A\times B)$ does not evaluate to ${\Omega}(A)\times{\Omega}(B)$ in some forcing extension; we show that forcing on ${\Omega}(X)$ itself will do. 
        To this end, we note that any generic filter $x$ on ${\Omega}(X)$ gives rise to an element of both $\widehat{{\Omega}(A)}$ and $\widehat{{\Omega}(B)}$. 
        Indeed, if $U\in x$ and $\cU$ is a family of open subsets of $A$ such that $A\cap\bigcup\cU=A\cap U$ then since $A$ is dense, $\{W\in{\Omega}(X)\mid \exists W^*\in\cU(W\cap A\subseteq W^*)\}$ is dense below $U$ so $\{U\cap A\mid U\in x\}$ is a point of $\widehat{A}$; 
        an identical proof holds for $\widehat{B}$. 
        Then $\{(I\cap A,J\cap B)\mid I,J\in{\Omega}(X), I\cap J=\varnothing\}$ is a cover of the space $A\times B$ which does not interpret to a cover of $\widehat{A}\times\widehat{B}$ as $(x,x)$ is in no interpretation of any element of this cover. 
        
    \end{proof}
    Note the same proof yields that any pair of dense subspaces of a space can always be forced to intersect. 
    One interesting point worth remarking is that each of these examples do not necessarily require us to pass to a large collapse extension; many in fact require adding only a single Cohen real. 

    The positive results of \cite{Zap} also transfer to results about locales. 
    As an immediate consequence of Corollary \ref{prod_cor} and \cite[Theorem 7.4]{Zap}, we obtain the following
    \begin{corollary}
        If $\langle X_i\mid i<\omega\rangle$ are interpretable in the sense of \cite{Zap} then $\bigoplus_{i<\omega}{\Omega}X_i$ is spatial. 
    \end{corollary}

    We now note some heavy-handed use of absoluteness allows us to extend the results of this section to all outer models, not simply forcing extensions. 
    For instance, 
    \begin{proposition} \label{prod_pres}
        Suppose $M$ is a transitive model of set theory and $M\models L,L'$ are locales. 
        Then the natural map $i\colon\widehat{L\oplus L'}\to\widehat{L}\times\widehat{L'}$ is a homeomorphism. 
    \end{proposition}
    \begin{proof}
        We first show $i$ is a surjection. 
        First, we note that $i(x)$ has as its first coordinate $\{\ell\in L\mid \downarrow(\ell,1)\in x\}$ and similarly for the second coordinate. 
        Since sets of the form $\downarrow(a,b)$ form a basis for $L\oplus L'$, whenever $x\in\widehat{L}$ and $y\in\widehat{L'}$, if it exists then the point in $\widehat{L\oplus L'}$ mapping to $(x,y)$ in $\widehat{L}\times \widehat{L'}$ must be the upwards closure of $\{\downarrow(\ell,\ell')\mid \ell\in x, \ell'\in y\}$. 
        What we must verify is that this is indeed a point of $\widehat{L\oplus L'}$; that is to say, this set defines a completely prime filter on $L\oplus L'$.
        
        Let $G$ be $V$-generic for $\operatorname{Coll}(\omega,(2^{|L|})^M+(2^{|L'|})^M)$. 
        Then in $M[G]$, since the product of locales interprets to the product of spaces in all forcing extensions, the assertion 
        \[\forall x\in\widehat{L}\;\forall y\in\widehat{L'}\;\forall\ell\in x\;\forall \ell'\in y\;\forall\cU\in L\oplus L'\;\left(\bigvee\cU=(\ell,\ell')\to\exists l_1\in x,l_2\in y,W\in\cU((l_1,l_2)\leq W)\right)\]
        holds as it asserts that any pair of points in $\widehat{L}$ and $\widehat{L'}$ induce a point in $L\oplus L'$. 
        By remark \ref{Borel}, this assertion is $\boldsymbol\Pi^1_1$ and therefore true in $V[G]$ by Mostowski absoluteness. 
        Note that any failure is upwards absolute, so it must hold in $V$; that is, $i$ is surjective. 

        For injectivity, suppose $x,y$ are distinct elements of $\widehat{L\oplus L'}$; since $\{(\ell,\ell')\mid \ell\in L,\ell'\in L'\}$ restricts to a cover of any set, we may find $(\ell,\ell')\in x\triangle y$; say $(\ell,\ell')\in x\setminus y$. 
        Then $i(x)\in\widehat{\ell}\times\widehat{\ell'}$ but $i(y)\not\in \widehat{\ell}\times\widehat{\ell'}$ so $x\neq y$.

        Finally, for any $\ell\in L$ and $\ell'\in L'$, the preimage of the basic open set $(\widehat{\ell},\widehat{\ell'})$ is $\widehat{(\ell,\ell')}$ and an identical argument to surjectivity yields that also the image of $\widehat{(\ell,\ell')}$ is $(\widehat{\ell},\widehat{\ell'})$. 
        As a bijective continuous open map, $i$ is a homeomorphism. 
    \end{proof}
    An analogous result holds more generally for limits in $M$, though some more elaborate formulae in the application of Mostowski absoluteness. 
    \begin{remark}
        At least for infinite products, Proposition \ref{prod_pres} requires that $M$ is well-founded; see \cite[Example 5.7]{Zap}.
    \end{remark}
    \begin{remark}
        It is not hard to show directly that if $(X,\pi)$ is a preinterpretation of $L$ and $(Y,\rho)$ is a preinterpretation of $L'$ then $(X\times Y,\sigma)$ is a preinterpretation of $L\oplus L'$, where $\sigma(U)=\bigcup_{(\ell,\ell')\in U}\widehat{\ell}\times\widehat{\ell'}$. 
        The main point in mapping joins to unions is that a single one sided join as in Definition \ref{prod_def} does not change the union and the join in $L\oplus L'$ can be constructed by transfinitely adding these one-sided joins. 
        A tedious argument along these lines should lead to a direct proof of Proposition \ref{prod_pres}.
    \end{remark}

\section{Globally defined spaces} \label{glob_section}
We now define a notion of spaces defined in all forcing extensions; examples include all locales, all universally Baire sets of reals, the interpretations of Hausdorff spaces of \cite{Frem}, and more. 
It is plausible that there are more axioms one should impose for what spaces defined in all forcing extensions should satisfy, though they should at least satisfy the following definition:
\begin{definition}
    A \emph{globally defined space} $X$ consists of 
    \begin{itemize}
        \item for each forcing $\mathbb{P}$, a name $\dot X_\mathbb{P}$ for a topological space;
        \item for each $\dot H$ a $\mathbb{P}$ name for a $\mathbb{Q}$ generic filter, a $\mathbb{P}$ name for a function $\dot f_{\mathbb{P},\dot H}^X\colon \dot X_\mathbb{Q}(\dot H)\to \dot X_{\mathbb{Q}}$
    \end{itemize}
    satisfying
    \begin{enumerate}
        \item if $G$ is the canonical name for an $\mathbb{P}$ generic filter then $\Vdash_{\mathbb{P}} f_{\mathbb{P},G}^X=\operatorname{id}$;
        \item $\mathbb{P}$ forces that $\dot f_{\mathbb{P},\dot H}^X$ is an embedding from the space with underlying set $\dot X_{\mathbb{Q}}(\dot H)$ with topology generated by the open sets in $V[H]$;
        \item the $f_{\mathbb{P},\dot H}^X$ are compatible in the sense that if $\dot H$ is an $\mathbb{P}$ name for a $\mathbb{Q}$ generic filter and $\dot I$ is a $\mathbb{Q}$ name for an $\mathbb{R}$ generic filter then $\mathbb{P}$ forces the following diagram commutes:
        \begin{center}
\begin{tikzcd}
                                                             & X_{\mathbb{P}} &                                                                                                                             \\
X_{\mathbb{Q}}(\dot H) \arrow[ru, "{f_{\mathbb{P}^X,\dot H}}"] &                & X_{\mathbb{R}}(\dot I(\dot H)) \arrow[lu, "{f_{\mathbb{P}^X,\dot I(\dot H)}}"'] \arrow[ll, "{f_{\mathbb{Q}^X,\dot I}(\dot H)}"]
\end{tikzcd}
        \end{center}
    \item for any forcing $\mathbb{P}$ and any $\mathbb{P}$ names $\dot{\mathbb{Q}},\dot{\mathbb{R}}$ for forcings, $\mathbb{P}\ast(\dot{\mathbb{Q}}\times\dot{\mathbb{R}})$ forces that the following is a pullback diagram for $G*(H_0\times H_1)$ the canonical decomposition of the generic filter:
    \begin{center}
\begin{tikzcd}
X_{\mathbb{P}}(\dot G) \arrow[rr] \arrow[dd] \arrow[dr, phantom, "\lrcorner", very near start] &  & X_{\mathbb{P}*\dot {\mathbb{Q}}}(\dot G*\dot H_0) \arrow[dd]  \\
                                             &{}  &                                             \\
X_{\mathbb{P}*\dot{\mathbb{R}}}(\dot G*\dot H_1) \arrow[rr]   &  & X_{\mathbb{P}*(\dot{\mathbb{Q}}\times\dot{\mathbb{R}})}(\dot G*(H_0\times H_1))
\end{tikzcd}
\end{center} \label{mutual_genericity}
If $X$ and $Y$ are globally defined spaces then a map $g\colon X\to Y$ consists of names for continuous functions $g_{\mathbb{P}}\colon X_{\mathbb{P}}\to Y_{\mathbb{P}}$ for each forcing $\mathbb{P}$ such that whenever $H$ is a $\mathbb{P}$ name for a $\mathbb{Q}$ generic filter, $\mathbb{P}$ forces that the following diagram commutes:
\begin{center}
\begin{tikzcd}
X_{\mathbb{P}}(\dot G) \arrow[r, "g_{\mathbb{P}}(\dot G)"]                                        & Y_{\mathbb{Q}}(\dot G)                                        \\
X_{\mathbb{Q}}(\dot H) \arrow[u, "{f_{\mathbb{P},\dot H}^X}"] \arrow[r, "g_{\mathbb{Q}}(\dot H)"] & Y_{\mathbb{Q}}(\dot H) \arrow[u, "{f_{\mathbb{P},\dot H}^Y}"]
\end{tikzcd}
\end{center}
We identify two such maps $g,h$ if for each $\mathbb{P}$, $\Vdash_{\mathbb{P}}g=h$. 

    \end{enumerate}
    
\end{definition}
\begin{remark}
    Item (\ref{mutual_genericity}) is an assertion of mutual genericity and asserts that the part of $X$ added by $G$ is the intersection of the parts added by $G*H_0$ and $G*H_1$. 
    Interestingly, this axiom excludes several interesting potential examples, including setting $X_\mathbb{P}$ to be $\omega_1^{V[G]}$ in the order topology: if both tail forcings collapse $\omega_1$ then the diagram fails to be a pullback. 
    It may be interesting to weaken item (\ref{mutual_genericity}) to require additional hypothesis on the tail forcings to allow for such examples. 
    None of our results depend on (\ref{mutual_genericity}) in any way.
\end{remark}
The main result of this section is the following
\begin{proposition}
    Interpretation gives a full, faithful, continuous embedding from locales to globally defined spaces.
\end{proposition}
\begin{proof}
    Faithfulness is immediate from item \ref{neq} of Proposition \ref{l_struct} and limit preservation follows from Proposition \ref{adj} plus limits of globally defined spaces being computed pointwise. 
    To see that interpretation is full, suppose $g\colon L\to M$ is a morphism of globally defined spaces. 
    Note that, when viewed as the space of completely prime filters over the ground model, each $f_{\mathbb{P},H}$ is the identity map. 
    In particular, for each $\mathbb{P}$ and each $\sigma\in\operatorname{Aut}(\mathbb{P})$, $\Vdash g_\mathbb{P}=\sigma(g_{\mathbb{P}})$. 
    Consider now $\mathbb{P}_\lambda=\operatorname{Coll}(\omega,\lambda)$ for some $\lambda\geq|L|+|M|$ and let $h_\lambda$ be the function 
    \[m\mapsto\{\ell\in L\mid \widehat{\ell}\subseteq g_{\mathbb{P}_\lambda}^{-1}\widehat{m}\}.\]
    Note that whenever $\sigma\in\operatorname{Aut}(\mathbb{P}_\lambda)$, the following diagram is forced to commute: 
    \begin{center}
\begin{tikzcd}
{\widehat{L}^{V[\sigma(G)]}} \arrow[rr, "\dot g_{\mathbb{P}}(\sigma(\dot G))"] &  & {\widehat{M}^{V[\sigma(G)]}}        \\
{\widehat{L}^{V[G]}} \arrow[rr, "\dot g_{\mathbb{P}}(G)"] \arrow[u, "="]       &  & {\widehat{M}^{V[G]}} \arrow[u, "="]
\end{tikzcd}
    \end{center}
    In particular, $\sigma^{-1}(\dot g_{\mathbb{P}_\lambda})( G)=\dot g_{\mathbb{P}_\lambda}(\sigma(G))=\dot g_{\mathbb{P}_\lambda}(G)$ so $h_{\lambda}$ is a fixed point of the action of $\operatorname{Aut}(\mathbb{P}_\lambda)$ on the space of names. By weak homogeneity of the Levy collapse, whenever $m\in M$ and $\ell\in L$, there is a $p$ forcing $\widehat{\ell}\in h_\lambda(m)$ if and only if every condition forces $\widehat{\ell}\in h_\lambda(m)$. 
    In particular, $h_\lambda$ is a ground model function and is independent of the generic filter. 
    Let $h_\lambda^*\colon L\to M$ have the corresponding frame map $m\mapsto\bigvee h_\lambda(m)$.
    \begin{claim}
        $h_\lambda^*$ is a locale map.
    \end{claim}
    \begin{proof}
        For preservation of arbitrary joins, suppose $\{m_i\mid i\in I\}$ are in $M$. 
        Trivially, $\bigvee_{i\in I}h_\lambda^*(m_i)\leq h_\lambda^*\left(\bigvee_{i\in I}m_i\right)$; suppose for contradiction the inequality was strict. 
        Then by item \ref{sym} of Proposition \ref{l_struct}, after forcing with $\mathbb{P}_\lambda$ there is an $x\in \widehat{h_\lambda^*\left(\bigvee_{i\in I}m_i\right)}\setminus \widehat{\bigvee_{i\in I}h_\lambda^*(m_i)}$. 
        Then $g_{\mathbb{P}_\lambda}(x)\in \widehat{\bigvee_{i\in I}m_i}$ but not in $\widehat{m_i}$ for any $i$, contradicting that $\{\widehat{m_i}\mid i\in I\}$ is a cover of $\widehat{\bigvee_{i\in I}m_i}$. 

        For preservation of finite meets, note that for the empty meet $1_L\in h_\lambda(1_M)$ so $h_\lambda^*(1_M)=1_L$. 
        For binary meets, whenever $m,m'\in M$, we see
        \[\begin{aligned}
            h_\lambda^*(m\wedge m')&=\bigvee h_\lambda(m\wedge m')\\
            &=\bigvee(h_\lambda(m)\cap h_\lambda(m'))\\
            &=\bigvee_{\substack{\ell\in h_\lambda(m)\\\ell'\in h_{\lambda}(m')}}\ell\wedge \ell'\\
            &=\left(\bigvee_{\ell\in h_{\lambda}(m)}\ell\right)\wedge \left(\bigvee_{\ell\in h_{\lambda}(m')}\ell\right)\\
            &=h_\lambda^*(m)\wedge h_\lambda^*(m')
        \end{aligned}\]
        for the first equality is the definition of $h_\lambda^*(m\wedge m')$, the second is observation that by definition of $h_\lambda$ that $h_\lambda(m\wedge m')=h_\lambda(m)\cap h_\lambda(m')$, the third is rewriting the terms of the join using that every set in the range of $h_\lambda$ is downwards closed, the fourth is distributivity coming from being a frame, and the last is the definition of $h_\lambda^*(m)$ and $h_\lambda^*(m')$.

    \end{proof}
    We now show that $h_\lambda^*$ induces $g$ for any $\lambda\geq|L|+|M|$. 
    Note first that $g_{\mathbb{P}_\lambda}=\widehat{h_\lambda^*}$.

Given any forcing $\mathbb{P}$, fix $\lambda\geq |\mathbb{P}|+|M|+|L|$. 
Then $\mathbb{P}_\lambda$ decomposes as a two step iteration $\mathbb{P}*\mathbb{P}_\lambda$; if $G*H$ is the name for the decomposition of the $\mathbb{P}_\lambda$ generic filter then $\mathbb{P}_\lambda$ forces that the following diagram commutes: 
\begin{center}
\begin{tikzcd}
{\widehat{L}^{V[G*H]}} \arrow[rr, "{\widehat{h_{\lambda}^*}^{V[G*H]}}"] && {\widehat{M}^{V[G*H]}}               \\
{\widehat{L}^{V[G]}} \arrow[rr, "g_{\mathbb{P}}(G)"] \arrow[u, hook]     && {\widehat{M}^{V[G]}} \arrow[u, hook]
\end{tikzcd}
\end{center}
In particular, $\Vdash_{\mathbb{P}_\lambda}g_{\mathbb{P}}(G)=\widehat{h_{\lambda}^*}$. 
Note in particular that if $\mathbb{P}=\operatorname{Coll}(\omega,\mu)$ for some $\mu\leq \lambda$ then $h_\mu^*=h_\lambda^*$ by item \ref{neq} of Proposition \ref{l_struct}. Since this holds for all $\mathbb{P}$, $g$ is indeed the map induced by $h_\lambda^*$ for some (equivalently all) sufficiently large value of $\lambda$. 
\end{proof}

\section{Interpretations of sufficiently complete spaces} \label{comp_section}

In this section, we show that completely metrizable spaces interpret to their metric completions under any appropriate metric and that locally compact Hausdorff spaces interpret to the space of germs at the generic filter. 
We first show that complete metric spaces interpret to complete metric spaces; this appears already as \cite[Corollary 5.5]{Zap} though we offer an alternative proof. 
We first provide a proof that holds only for forcing extensions and then note that some use of Mostowski absoluteness allows us to generalize to all transitive models of set theory. 

\begin{proposition} \label{complete_met}
    Suppose $X$ is completely metrizable and $d$ is any complete metric on $X$. 
    $\widehat{{\Omega}X}$ is homeomorphic to the $d$-completion of $X$ in any forcing extension. 
\end{proposition}
\begin{proof}
    We denote the $d$-completion of $X$ by $\overline{(X,d)}$ and work in some forcing extension $V[G]$. 
    For each $\mathcal{F}\in\widehat{{\Omega}X}$ and $\alpha<\omega$, pick $x_\alpha^{\mathcal{F}}$ such that $B_{1/\alpha}(x_\alpha^{\mathcal{F}})\in\mathcal{F}$. 
    Note that since $\mathcal{F}$ is a filter, for each $\alpha,\beta$, $B_{1/\alpha}(x_\alpha^{\mathcal{F}})\cap B_{1/\beta}(x_\beta^{\mathcal{F}})\neq\varnothing$. 
    In particular, $d(x_\alpha,x_\beta)<\frac1\alpha+\frac1\beta$ so $(x_\alpha^{\mathcal{F}})_{\alpha<\omega}$ is $d$-Cauchy. Let $f\colon\widehat{X}\to\overline{(X,d)}$ be the function $f(\mathcal{F})=[(x_\alpha^{\mathcal{F}})_{\alpha<\omega}]$. 
    Note that $f(\mathcal{F})$ does not depend on the choice of $x_\alpha^{\mathcal{F}}$ since if $B_{1/\alpha}(x),B_{1/\alpha}(y)\in\mathcal{F}$ then since $\mathcal{F}$ is a filter, $B_{1/\alpha}(x)\cap B_{1/\alpha}(y)\neq\varnothing$ and $d(x,y)<\frac2\alpha$.
    We claim $f$ is a homeomorphism. 
    
    First, we show 
    \begin{claim}
        $f$ is continuous. 
    \end{claim}
    \begin{proof}
    Given $\mathcal{F}\in\widehat{X}$ and rational $\varepsilon>0$, let $\alpha=\lceil\frac{4}{\varepsilon}\rfloor$. 
    Then $B_{1/\alpha}(x_\alpha^\mathcal{F})\in\mathcal{F}$ and whenever $B_{1/\alpha}(x_n^\mathcal{F})\in\mathcal{G}$, we see $d(x_\alpha^{\mathcal{F}},x_\alpha^{\mathcal{G}})<2/\alpha$. 
    In particular, 
    \[\begin{aligned}
        \widehat{d}(f(\mathcal{F}),f(\mathcal{G}))&\leq \widehat{d}(f(\mathcal{F}),x_\alpha^{\mathcal{F}})+\widehat{d}(x_\alpha^{\mathcal{F}},x_\alpha^{\mathcal{G}})+d(x_\alpha^{\mathcal{G}}, f(\mathcal{G}))\\
        &<\frac1\alpha+\frac{2}{\alpha}+\frac1\alpha\\
        &\leq\varepsilon
    \end{aligned}\]
    as desired. 
    \end{proof}

    We now construct the inverse to $f$. 
    Given $x=[(x_\alpha)]$, let $g(x)=\{U\in V\mid\exists\varepsilon>0\exists N\forall \alpha>N(B_\varepsilon(x_\alpha)\subseteq U)\}$; that is, $g(x)$ is all those ground model open sets $U$ that some (equivalently any) Cauchy sequence converging to $x$ is eventually further than some $\varepsilon$ away from the complement of $U$; note this definition does not depend on the choice of representative. 
    \begin{claim}
        $g(x)\in\widehat{X}$. 
    \end{claim}
    \begin{proof}
        This is a density argument. Given names $\dot x_\alpha$ for a Cauchy sequence and a condition $p$ forcing that $U\in g([(\dot x_\alpha)])$ with witness $\varepsilon$ and an open cover $\mathcal{U}$ of $U$, find $p\geq p_0\geq p_1\geq\ldots$ such that 
        \begin{itemize}
            \item $p_\alpha$ decides $\dot x_\alpha$ to be some $y_\alpha\in X$
            \item for some $N_\alpha$, $p_\alpha\Vdash\forall \beta,\gamma>N_\alpha d(\dot x_\beta,x_\gamma)<\varepsilon$. 
        \end{itemize}
        Then $(y_\alpha)$ is $d$-Cauchy and so converges to some $y$ by completeness of $(X,d)$. 
        Then $d(y,X\setminus U)\geq \varepsilon$ so $y\in U$. 
        Since $\mathcal{U}$ is a cover of $U$, let $W\in \cU$ and $\delta>0$ be such that $B_\delta(y)\subseteq W$. 
        If $\alpha>\max(4/\delta,N_{4/\delta})$ then $p_\alpha$ forces that $W\in g(x)$ as $p_\alpha\Vdash d(x_\beta,x_\alpha)\leq\frac{\delta}{4}$ for all $\beta>\alpha$ and $B_{\delta/2}(x_\alpha)\subseteq W$. 
    \end{proof}
    \begin{claim}
        $g$ is continuous
    \end{claim}
    \begin{proof}
        Given any $x=[(x_\alpha)]\in\overline{(X,d)}$ and any basic open $\widehat{U}\ni g(x)$, there is an $\varepsilon>0$ such that for sufficiently large $\alpha$, $d(x_\alpha,X\setminus U)>\varepsilon$. 
        Given any $y=[(y_\alpha)]$, if $\widehat{d}(x,y)<\varepsilon/2$ then for sufficiently large $\alpha$, $d(y_\alpha,X\setminus U)<\varepsilon/2$ so $g(y)\in\widehat{U}$, as desired. 
    \end{proof}
    \begin{claim}
        $f$ and $g$ are inverses to each other.
    \end{claim}
    \begin{proof}
        We first show that for any $\mathcal{F}\in\widehat{X}$, $gf(\mathcal{F})=\mathcal{F}$. Given any $\mathcal{F}\in\widehat{X}$ and any $U\in \mathcal{F}$, $\mathcal{V}_U=\{W\subseteq U\mid d(W,X\setminus U)>0\}$ is an open cover of $U$ so there is a $W\in\mathcal{F}\cap\mathcal{V}_U$. 
        If $\alpha$ is sufficiently large that $d(W,X\setminus U)>2/\alpha$ then since $W\cap U_n^\mathcal{F}\neq\varnothing$, 
        $d(x_\alpha,X\setminus U)>d(W,X\setminus U)-1/\alpha>1/\alpha$ so $U\in gf(\mathcal{F})$. 

        Conversely, suppose $U\in gf(\mathcal{F})$. Fix $\varepsilon>0$ such that for all sufficiently large $\alpha$, $d(x_\alpha^{\mathcal{F}},X\setminus U)>\varepsilon$. 
        Then if $\alpha>2/\varepsilon$, $B_{1/\alpha}(x_\alpha^\mathcal{F})\subseteq U$ so $U\in\mathcal{F}$. 

        We now show that $fg(x)=x$ for $x\in\overline{(X,d)}$. 
        Given $x=[(x_\alpha)]\in\overline{(X,d)}$ and $\gamma<\omega$, since $(x_\alpha)$ is $d-Cauchy$ we may find an $N$ such that if $\alpha,\beta>N$ then $d(x_\alpha,x_\beta)<1/\gamma$. 
        In particular, if $\alpha>N$ then $B_{2/\gamma}(x_\alpha)\in g(x)$ so $d(x_\alpha,x_\alpha^{g(x)})\leq\frac1\gamma+\frac1\alpha$. 
        In particular, $(x_\alpha)_{\alpha<\omega}$ is equivalent to $(x_\alpha^{g(x)})_{\alpha<\omega}$.
    \end{proof}
\end{proof}
We note that with some mild absoluteness, the proof generalizes to all transitive models.
\begin{corollary} \label{complete_wf}
    Suppose $M$ is a transitive model of set theory and $M\models (X,d)$ is a complete metric space. 
    Then $\widehat{{\Omega}(X)}$ is homeomorphic to the $d$ completion of $X$
\end{corollary}
\begin{proof}
    The assertion that for every $d$-Cauchy sequence $(x_n)$, $\{U\in V\mid\exists\varepsilon>0\exists N\forall n>N(B_\varepsilon(x_n)\subseteq U)\}\in\widehat{{\Omega}(X)}$ is downwards absolute and $\boldsymbol\Pi^1_1$ in any model where $(2^{{\Omega}(X)})^M$ is countable and so is true in $V$ by Mostowski absoluteness. 
    Proving this assertion was the only use of forcing used in the proof of Proposition \ref{complete_met}
\end{proof}
\begin{remark}
    Well-foundedness of $M$ is essential for Corollary \ref{complete_wf}; see \cite[Example 5.7]{Zap}. 
\end{remark}

We now analyze the situation of interpreting locally compact Hausdorff spaces and show directly that $\widehat{{\Omega}X}$ coincides with the constructions of \cite{Frem}. 
The intuition for the construction of \cite{Frem} is to define the interpretation of a space $X$ via germs of continuous functions on the Stone space $S(\mathbb{B})$ of a complete Boolean algebra at the generic filter. 
Formally, for a complete Boolean algebra $\mathbb{B}$, we define $C_-(X)$ as the set of continuous functions defined on a comeager subspace of $S(\mathbb{B})$. 
The Fremlin interpretation of $X$ has underlying set $C_-(X)/~_G$ where $f~_Gy$ if and only if there is $b\in G$ such that $f\upharpoonright N_b=^*g\upharpoonright N_b$ for $N_b$ the basic open set of filters containing $b$ and $=^*$ meaning equality off a meager set. 
The Fremlin interpretation is topologized by declaring that for each $U\subseteq X$ open in the ground model, $\{[f]\mid\exists b\in G(f[N_b]\subseteq^* U)\}$ is open. 
We note that in full generality, our interpretation does not coincide with the Fremlin interpretation: Example \ref{non_Haus} gives a Hausdorff space whose interpretation is not Hausdorff whereas by \cite[2A(c)]{Frem}, the Fremlin interpretation of any Hausdorff space is Hausdorff. 

We note that by \cite[Theorem 16.1]{Zap}, the Fremlin interpretation coincides with ours for interpretable spaces; we here only show they coincide for locally compact spaces but the more general case can be proved directly by much the same but more cumbersome methodology. 
In light of Proposition \ref{pt-map}, we need only to prove

\begin{proposition}
    Suppose $X$ is locally compact Hausdorff, $\mathbb{B}$ is a complete Boolean algebra, and $f\colon \mathcal{O}X\to\mathbb{B}$ is a map of frames. 
    There is a dense open $U\subseteq S(\mathbb{B})$ and a continuous $g\colon U\to X$ such that whenever $W\subseteq X$ is open, $N_{f(W)}=\overline{g^{-1}W}$. 
    Moreover, if $h\colon U'\to X$ is any other function satisfying the same hypotheses then $h\upharpoonright U\cap U'=g\upharpoonright U\cap U'$. 
\end{proposition}
\begin{proof}
    Given $f$, let $\cU$ be an open cover of $X$ such that whenever $W\in\cU$, $\overline{W}$ is compact. 
    Then $U=\{H\in S(\mathbb{B})\mid\exists W\in\cU(f(W)\in H)\}$ is open and since $\bigvee_{W\in\cU}f(W)=1$, $U$ is dense. 
    For $H\in S(\mathbb{B})$, let $F_H=\{\overline{U}\mid f(U)\in H\}$; note that since $F_H$ is a family of closed sets with the finite intersection property and that some element of $f_H$ is compact, $\bigcap f_H\neq\varnothing$. 
    Moreover, $\bigcap f_H$ has exactly one point: if $x\neq y$, since $X$ is locally compact Hausdorff we may find open $A,B$ such that $x\in A\subseteq\overline{A}\subseteq B$ and $y\not\in B$. 
    Then since $f$ is a frame map, $f(B)\vee f(X\setminus\overline{A})=1$. 
    In particular, either $\overline{B}\in F_H$ or $\overline{X\setminus\overline{A}}\in F_H$; in the first case, $y$ is not in $\bigcap F_H$ and in the second, $x$ is not in $\bigcap F_H$. 

    Let $g(H)$ be the unique point in $\bigcap F_H$. 
    We show $g$ is continuous. 
    Suppose $H\in U$ and $W\supseteq f(H)$ is open; by shrinking $W$ if necessary, we may assume $\overline{W}$ is compact. For each $x\in \overline{W}\setminus W$, fix an open $W_x\ni g(H)$ such that $x\not\in \overline{W_x}$ and $f(W_x)\in H$. 
        Then $\{\overline{W}\setminus W_x\}$ is a cover of the compact space $\overline{W}\setminus W$ so there are $x_0,\ldots,x_n$ such that $\overline{W}\setminus W\subseteq \bigcup_i\overline{W}\setminus W_{x_i}$. 
        Note that $f(W)\vee \bigvee_if(X\setminus\overline{W_{x_i}})\vee f(X\setminus\overline{W})=1$ since $f$ is a map of frames; note in particular this implies $f(W)\in H$. 
        In particular, $f(W\cap\bigcap_iW_{x_i})=f(W)\wedge\bigwedge f(W_{x_i})\in H$. 
        Moreover, $f\left(W\cap\bigcap_i W_{x_i}\right)=\bigwedge_i f(W_{x_i})\in H$ and if $\bigwedge_if(W_{x_i})\in H'$ then $g(H')\in \overline{W}\cap\bigcap_i\overline{W_{x_i}}\subseteq W$. 
        
    As noted in the preceding paragraph, \[g^{-1}(U)= \bigcup g^{-1}\{W\subseteq U, \overline W\text{ compact}\}\subseteq \bigcup_{\{W\subseteq U, \overline W\text{ compact}\}}N_{f(W)}\subseteq N_{f(U)}\] for each $U\subseteq X$ open. 
    To see that $N_{f(U)}\subseteq\overline{g^{-1}(U)}$, fix by local compactness of $X$ a cover $\cU$ of $U$ such that whenever $W\in \cU$, $\overline{W}\subseteq U$. 
    Then $\bigcup \cU=U$ so $f(U)=\bigvee\{f(W)\mid W\in \cU\}$ so that 
    \[\bigcup_{W\in\cU}N_{f(W)}\subseteq g^{-1}U\subseteq N_{f(U)}\]
    is dense. 

    Given any other such $h$, whenever $U\ni g(H)$ is open, 
    \[H\in g^{-1}U\subseteq \overline{g^{-1}U}=N_{f(U)}=\overline{h^{-1}U}\subseteq h^{-1}\overline{U}\cup (S(\mathbb{B})\setminus\operatorname{dom}(h)).\]
    So if $h(H)$ is defined then $h(H)\in\bigcap\{\overline{U}\mid f(U)\in H\}=\bigcap F_H$. 
    But $\bigcap F_H$ has a unique point $g(H)$ so $g(H)=h(H)$. 
\end{proof}
As a nice non-Hausdoff example, we show that the spectrum of a commutative ring with $1$ interprets to the spectrum of the same commutative ring with $1$; note that more prime ideals may be added by forcing. 
First recall the basic definition. 
\begin{definition}
    Suppose $R$ is a commutative ring with $1$. 
    $\operatorname{Spec}(R)$ is the space whose points are prime ideals in $R$ with the topology generated by declaring that for each $a\in R$, $O_a=\{P\mid a\not\in P\}$ is open. 
\end{definition}
We show the following; note that unlike the previous examples we do not need to assume we are in a forcing extension:
\begin{proposition}
    $M\models ZFC$ and $R\in M$ is a commutative ring with $1$. 
    Let $L$ be the locale of open subsets of $\operatorname{Spec}(R)$ as computed in $M$. 
    Then $\widehat{L}$ is naturally homeomorphic to $\operatorname{Spec}(R)$. 
\end{proposition}
\begin{proof}
    Suppose $\mathcal{F}\subseteq {\Omega}\operatorname{Spec}(R)$ is completely prime over $M$ and let $P_{\mathcal{F}}=\{a\in R\mid O_a\not\in\mathcal{F}\}$. 
    \begin{claim}
        $P_{\mathcal{F}}$ is a prime ideal in $R$.
    \end{claim}
    \begin{proof}

        To see that $P_{\mathcal{F}}$ is an ideal, first note $\operatorname{Spec}(R)=O_1\in\mathcal{F}$ so $1\not\in P_{\mathcal{F}}$ and $\varnothing=O_0\not\in\mathcal{F}$ so $0\in P_{\mathcal{F}}$. 
        
        For closure under addition, if $a,b\in P_\mathcal{F}$ then since $O_{a+b}\subseteq O_a\cup O_b$ and $O_a,O_b\not\in\mathcal{F}$, we must have $O_{a+b}\not\in\mathcal{F}$ since $\mathcal{F}$ is a filter. 
        In particular, $a+b\in P_{\mathcal{F}}$. 

        For closure under multiplication by ring elements, if $a\in P_{\mathcal{F}}$ and $b\in R$ then since $O_{ab}\subseteq O_a$, $ab\in P_\mathcal{F}$ since $\mathcal{F}$ is a filter. 

        Finally, if $O_{ab}\not\in\mathcal{F}$ then since $O_a\wedge O_b=O_{ab}$ and $\mathcal{F}$ is a filter, we must have either $O_a\not\in\mathcal{F}$ or $O_b\not\in\mathcal{F}$. 
    \end{proof}
    For the inverse, given a prime ideal $P$ let $\mathcal{F}_P$ be the filter generated by $\{O_a\mid a\not\in P\}$. 
    \begin{claim}
        $\mathcal{F}_P$ is completely prime over $M$. 
    \end{claim}
    \begin{proof}
        Note that trivially, $\mathcal{F}_P$ is upwards closed, $1\in\mathcal{F}_P$, and $0\not\in\mathcal{F}_p$. 
        For closure under finite meets, if $O_a,O_b\in\mathcal{F}_p$ then $a,b\not\in P$; since $P$ is prime, $ab\not\in P$ so $O_{ab}\in\mathcal{F}_p$. 
        Finally, suppose $U\in\mathcal{F}_p$ and $\cU\in M$ such that $\bigvee\cU=U$. 
        By shrinking $U$ and $\cU$ if necessary, we may assume $U$ and all elements of $\cU$ are basic open sets; let $U=O_a$ and $\cU=\{O_{b_i}\mid i\in I\}$. 
        Then every prime ideal in $M$ containing each $b_i$ also contains $a$; since $M\models ZFC$, $a\in\sqrt{\langle b_i\rangle}$. 
        That is, there are $n<\omega$, $r_0,\ldots,r_{n-1}\in R$, $i_0,\ldots,i_{n-1}\in I$, and $0<e_0,\ldots,e_{n-1}<\omega$ such that $a^n=\sum r_ib_i^{e_i}$. 
        Since $P$ is a prime ideal in $V$, $a\not\in P$, and this equation still holds in $V$, we must have some $b_i$ not be in $P$; that is, $O_{b_i}\in\mathcal{F}_P$. 
        
    \end{proof}
    Note that $\mathcal{F}\mapsto P_\mathcal{F}$ and $P\mapsto\mathcal{F}_P$ are both continuous. 
    What remains to show is that for $\mathcal{G}$ completely prime over $M$, $\mathcal{F}_{P_{\mathcal{G}}}=\mathcal{G}$ and for $P$ a prime ideal, $P_{\mathcal{F}_P}=P$. 

    Indeed, if $P$ is a prime ideal then 
    \[P_{\mathcal{F}_P}=\{a\mid O_a\not\in\mathcal{F}_P\}=P.\]
    If $\mathcal{G}$ is completely prime over $M$ and $O\in\mathcal{G}$, we may find some $a$ such that $O_a\subseteq O$ and $O_a\in\mathcal{G}$. 
    Then $a\not\in P_{\mathcal{G}}$ so $O\supseteq O_a\in \mathcal{F}_{P_{\mathcal{G}}}$. 
    Conversely, if $O\in \mathcal{F}_{P_{\mathcal{G}}}$ we may find some $a\not\in P$ such that $O_a\subseteq O$. 
    Then $O_a\in\mathcal{G}$ so $O\in\mathcal{G}$.
\end{proof}

\section{Separation Properties of Interpretations} \label{sep_section}
We now begin analyzing what properties of a locale $L$ are necessary for $\widehat{L}$ to have nice separation properties. 
For the context of $T_2$ and $T_3$, we show that $\widehat{L}$ always having nice separation properties is equivalent to the localic notions of $L$ being Isbell Hausdorff and $L$ being regular respectively. 
We first note that there are indeed Hausdorff spaces that interpret to spaces that are not even $T_1$; our example is the usual example of a Hausdorff space whose corresponding locale is not Isbell Hausdorff (see \cite[III.1.5]{Johnstone}). 

\begin{example} \label{non_Haus}
    Let $X$ be the space with underlying set $\mathbb{R}$ with the topology generated by the usual (Euclidean) topology together with $\{\mathbb{Q}\}$. 
    Recall that $x$ is a Cohen real over $V$ if $x$ is in every dense open subset of $\mathbb{R}$ which is coded in $V$; we show that any forcing that adds a Cohen real will force that $\widehat{X}$ is not Hausdorff. 
    Suppose in $V[G]$ that $x$ is a Cohen real over $V$. 
    Let $\mathcal{F}_1=\{U\cup (W\cap\mathbb{Q})\mid x\in U\}$ and $\mathcal{F}_2=\{U\cup (W\cap\mathbb{Q})\mid x\in W\}$. 
    Then $\mathcal{F}_1$ and $\mathcal{F}_2$ are both points of $\widehat{X}$: for $\mathcal{F}_1$ this uses no properties of $x$ besides being irrational. 
    For $\mathcal{F}_2$ this is derived easily from noting that if $\cU$ is an open cover of $\mathbb{Q}$ and $\mathcal{V}$ is a family of open sets such that $\cU=\{W\cap\QQ\mid W\in\mathcal{V}\}$ then $\bigcup\mathcal{V}$ is dense open so $x\in\bigcup\mathcal{V}$. 
    
    Then $\mathcal{F}_1\subsetneq\mathcal{F}_2$ so every basic open set containing $\mathcal{F}_1$ also contains $\mathcal{F}_2$ but $\mathcal{F}_1\neq\mathcal{F}_2$. 
\end{example}

We first prove implications about the $T_1$ axiom, simply asserting that points are closed. 
We here see a nice link between localic notions and the corresponding notions for spaces in forcing extensions: item (\ref{Tnu}) is referred to as the axiom $T_U$ in, for instance, \cite[III.1.5]{Johnstone} (short for totally unordered).    

\begin{proposition} \label{T1_equiv}
    The following are equivalent for a locale $L$: 
    \begin{enumerate}
        \item $\widehat{L}$ is $T_1$ in all forcing extensions. \label{T1all}
        \item for all locales $L'$ and all frame maps $f,g\colon L\to L'$ with $f\leq g$ pointwise, $f=g$. \label{Tnu}
        \item for all complete Boolean algebras $\mathbb{B}$ and all frame maps $f,g\colon L\to\mathbb{B}$ with $f\leq g$ pointwise, $f=g$. \label{T1BA}
    \end{enumerate}
\end{proposition}
\begin{proof}
    Note that $(\ref{Tnu})\Rightarrow(\ref{T1BA})$ is trivial. 
    
    We show $(\ref{T1all})\Rightarrow(\ref{Tnu})$ by contraposition. 
    Suppose $f,g\colon L\to L'$ are distinct maps of frames with $f\leq g$ pointwise and work in $V[G]$ for $G$ $V$-generic for $\operatorname{Coll}(\omega,\max(|L|,|L'|))$.
    By item \ref{neq} of Proposition \ref{l_struct}, $\widehat{f}\neq \widehat{g}$ so we may fix some $x\in\widehat{L'}$ such that $\widehat{f}(x)\neq\widehat{g}(x)$. 
    Since $f\leq g$ pointwise, $\widehat{f}(x)\subseteq\widehat{g}(x)$; that is, every basic open set containing $\widehat{f}(x)$ also contains $\widehat{g}(x)$ so that $\widehat{L}$ is not $T_1$. 

    We finally show $(\ref{T1BA})\Rightarrow(\ref{T1all})$ by contraposition. 
    Suppose $\mathbb{B}$ is a complete Boolean algebra forcing that $L$ is not $T_1$ and let $x,y$ be $\mathbb{B}$ names for distinct elements of $\widehat{L}$ such that $y\in\overline{x}$. 
    Let $f,g\colon L\to\mathbb{B}$ be the corresponding frame maps from Proposition \ref{pt-map}. 
    Then whenever $\ell\in L$, $f(\ell)=\llbracket x\in\widehat{\ell}\rrbracket\leq\llbracket y\in\widehat{\ell}\rrbracket=g(\ell)$ so $f,g$ witness a failure of (\ref{T1BA}).
\end{proof}
An easy use of Mostowski absoluteness allows us to extend beyond forcing extensions:
\begin{corollary} \label{Most_T1}
    Suppose $M$ is a transitive model of set theory and $M\models L$ is a $T_U$ locale. 
    Then $\widehat{L}^V$ is $T_1$.
        
\end{corollary}

\begin{proof}
    Let $G$ be $V$-generic for $\operatorname{Coll}(\omega,(2^{|L|})^M)$. 
    Then by Proposition \ref{T1_equiv}, $\widehat{L}^{M[G]}$ is $T_1$. 
    Note that $\widehat{L}$ being $T_1$ is the assertion that 
    \[\forall x,y\in \widehat{L}\exists\ell\in L(\ell\not\in x\wedge \ell\in y);\]
    by Remark \ref{Borel}, this is assertion is $\Pi^1_1(L,\mathcal{P}(L)\cap M)$ so by Mostowski absoluteness, $\widehat{L}^{V[G]}$ is $T_1$. 
    Since a failure for an interpretation to be $T_1$ is upwards absolute, we must have that $\widehat{L}^V$ is $T_1$.
\end{proof}

We now investigate when a locale interprets to a Hausdorff space. 

\begin{proposition} \label{T2_equiv}
The following are equivalent for a locale $L$: 
    \begin{enumerate}
        \item $\widehat{L}$ is $T_2$ in all forcing extensions. \label{T2all}
        \item $\widehat{L}$ is $T_2$ after forcing with $\operatorname{Coll}(\omega,{2^{|L|}})$. \label{T2coll}
        \item $L$ is $I$-Hausdorff. \label{IHaus}
        \item For all locales $L'$ and all frame maps $f,g\colon L\to L'$ with $f\neq g$, there are $\ell,\ell'\in L$ such that $\ell\wedge\ell'=0$ but $f(\ell)\wedge g(\ell')\neq 0$. \label{T2map}
        \item For all complete Boolean algebras $\mathbb{B}$ and all frame maps $f,g\colon L\to \mathbb{B}$ with $f\neq g$, there are $\ell,\ell'\in L$ such that $\ell\wedge\ell'=0$ but $f(\ell)\wedge g(\ell')\neq 0$. \label{T2mapBA}.
    \end{enumerate}
\end{proposition}
\begin{proof}
    Note that $(\ref{T2all})\Rightarrow (\ref{T2coll})$ and $(\ref{T2map})\Rightarrow(\ref{T2mapBA})$ are trivial. 
    We prove $(\ref{T2coll})\Rightarrow(\ref{IHaus})$, $(\ref{IHaus})\Rightarrow(\ref{T2map})$, and $(\ref{T2mapBA})\Rightarrow(\ref{T2all})$.

    We first prove $(\ref{T2coll})\Rightarrow(\ref{IHaus})$ by contraposition. 
    Recall by Proposition \ref{adj} that interpretation preserves products of locales. 
Now recall that $L\oplus L$ consists of downsets in $L\times L$ which are closed under taking one sided joins with the other coordinate fixed. 
Moreover, letting $d_L=\{(U,V)\mid U\wedge V=\varnothing\}$, $L$ is $I$-Hausdorff if and only if whenever $U\supseteq d_L$ is in $L\oplus L$ and $(a\wedge b,a\wedge b)\in U$ then $(a,b)\in U$. 
By $\neg(\ref{IHaus})$, we may fix $U\supseteq d_L$ saturated and $a,b\in L$ such that $(a\wedge b,a\wedge b)\in U$ but $(a,b)\not\in U$; in particular, $(\downarrow(a,b))\not\leq U$. 

Suppose $G$ is generic for $\operatorname{Coll}(\omega,{2^{|L|}})$. Since $|L\oplus L|\leq 2^{|L|}$, by item \ref{sym} of Proposition \ref{l_struct}, there is a $z\in\widehat{\downarrow(a,b)}\setminus \widehat{U}$. 
Then fix $x,y\in \widehat{L}$ such that $z=(x,y)$. 
We claim that $x\neq y$ but that $x$ and $y$ cannot be separated by open sets. 
First, observe $a\in x$ and $b\in y$. 
However, since $\downarrow(a\wedge b,a\wedge b)\not\in z$ (as otherwise $z\in\widehat{U}$), we must have either $b\not\in x$ or $a\not\in y$ so that $x\neq y$. 
To see that $x$ and $y$ cannot be separated by open sets, suppose $x\in\widehat{c}$ and $y\in\widehat{d}$. 
Then $\downarrow(c,d)\in z$ so $c\wedge d\neq 0$ as $U\supseteq d_L$ and $z\not\in U$. 
By item \ref{nonz} of Proposition \ref{l_struct}, $\widehat{c}\wedge\widehat{d}=\widehat{c\wedge d}\neq\varnothing$, as desired and completing the proof of $\neg(\ref{T2coll})$. 

We now show $(\ref{IHaus})\Rightarrow \ref{T2map}$. Fix $f,g\colon L\to L'$ such that whenever $\ell\wedge \ell'=0$, we have $f(\ell)\wedge g(\ell')$; we show $f=g$. 
Consider $f\oplus g\colon L\oplus L\to L'$. 
     By hypothesis, whenever $\ell\wedge\ell'=0$, $f\oplus g(\ell\oplus\ell')=0$. 
     Since $f\oplus g$ preserves arbitrary joins, $f\oplus g(d_L)=0$. 
     In particular, for any $a\in L$, $f\oplus g(a)\cong f\oplus g(a\vee d_L)$. 
     Since $L$ is $I$-Hausdorff, any element of $L\oplus L$ above $d_L$ is of the form $(\operatorname{id}\oplus\operatorname{id})_*\ell$ for some (unique) $\ell\in\ell$. 
     In particular, $f\oplus g$ factors through $\operatorname{id}\oplus\operatorname{id}$ via some map $h$. 
     Then $h\oplus h=f\oplus g$ so $f=g=h$. 

    Finally, $(\ref{T2mapBA})\Rightarrow(\ref{T2all})$ is a typical density argument. 
    Suppose we force with a complete Boolean algebra $\mathbb{B}$, $b\in\mathbb{B}$, and $x,y$ are names for distinct points of $\widehat{L}$. 
    As in Proposition \ref{pt-map}, $x$ and $y$ correspond to frame maps $f,g\colon L\to\downarrow b$ respectively. By (\ref{T2mapBA}), there are $\ell,\ell'\in L$ such that $\ell\wedge\ell'=0$ but $f(\ell)\cap g(\ell')\neq 0$. 
    Then $f(\ell)\wedge g(\ell')$ forces that $\ell\in x$, $\ell'\in y$, and no point of $\widehat{L}$ contains both $\ell$ and $\ell'$. 
    In particular, the basic open sets determined by $\ell$ and $\ell'$ separate $x$ and $y$.

\end{proof}
A nearly identical analysis to Corollary \ref{Most_T1} yields the following:
\begin{corollary} \label{most_haus}
    Suppose $M$ is a transitive model of set theory and $M\models L$ is an $I$-Hausdorff locale. 
    Then $\widehat{L}^V$ is a Hausdorff space. 
\end{corollary}
\begin{proof}
    Suppose $L$ is $I$-Hausdorff in $M$ and let $G$ be $V$-generic for $\operatorname{Coll}(\omega,(2^{|L|})^M)$. 
    In $M[G]$, the $\Pi_1^1(L,\mathcal{P}(L)\cap M)$ assertion
    \[\forall x,y\in\widehat{L}\exists \ell,\ell'\in L(\ell\wedge\ell'=0, \ell\in x, \ell'\in y)\]
    is true, so the same assertion is true in $V[G]$ by Mostowski absoluteness. 
    Since this formula is downwards absolute, it also holds in $V$ and implies that $\widehat{L}^V$ is a Hausdorff space. 
\end{proof}

We now investigate the axiom $T_3$. 
For $\ell,\ell'\in L$, we write $\ell'\prec\ell$ and say $\ell'$ is well below $\ell$ if there is a $\ell''$ such that $\ell'\wedge\ell''=0$ and $\ell\vee\ell''=1$. 
Recall that a locale $L$ is \emph{regular} if for every $\ell\in L$, $\ell=\bigvee\{\ell'\mid \ell'\prec\ell\}$. 
Note this definition is a natural extension of the spatial one: a $T_0$ space $X$ is regular if and only if whenever $x\in X$ and $U\subseteq X$ is open, there is an open $W\ni x$ such that $\overline{W}]\subseteq U$. 
\begin{proposition}
    The following are equivalent:
    \begin{enumerate}
        \item $L$ is regular; \label{reg}
        \item $\widehat{L}$ is $T_3$ in all forcing extensions. \label{T3force}
    \end{enumerate}
\end{proposition}

    \begin{proof}
        For $(\ref{reg})\Rightarrow(\ref{T3force})$, first recall that by \cite[Proposition 5.4.2]   {FrmLoc} every regular locale is $I$-Hausdorff so $\widehat{L}$ is $T_2$ in all forcing extensions. Moreover, note a $T_1$ space with a distinguished basis is $T_3$ if and only if for every basic open $U$, there is a cover $\mathcal{U}$ of $U$ such that whenever $W\in\mathcal{U}$, $\overline{W}\subseteq U$. 
        But the condition of a regular locale is precisely that this assertion occurs with $\mathcal{U}$ a collection of basic open sets coded in the ground model.

        Conversely, suppose $L$ is not regular and fix some $\ell\in L$ which is not the join of conditions well below $\ell$. 
        Let $\mathbb{P}_\ell=\{a\in L\mid a\neq\bigvee\{b\mid b\prec\ell, b\leq a\}\}$ with the order inherited from $L$. 
        Note that if $a\in\mathbb{P}_\ell$ and $\bigvee\cU=a$ then there is some $b\in\cU\cap\mathbb{P}_L$ so the upwards closure of the generic for $\mathbb{P}_L$ is a point of $\widehat{L}$. 
        We force with $\mathbb{Q}=\mathbb{P}_\ell\times\operatorname{Coll}(\omega,2^{|L|})$; 
        note that $\mathbb{Q}$ is equivalent to $\operatorname{Coll}(\omega,|L|)$. We claim that if $x$ is the generic point added by $\mathbb{P}_\ell$ cannot be separated from $\widehat{L}\setminus\widehat{\ell}$. 
        Indeed, suppose some $(a,p)$ forces that $b\in L$ determines a basic open set containing $x$ with $\overline{\widehat{b}}\subseteq\widehat{\ell}$. 
        By extending $a$ if necessary, we may assume that $a\leq b$. 
        We claim $a$ itself is well below $\ell$ as witnessed by $d:=\bigvee\{c\mid a\wedge c=0\}$, contrary to the definition of $\mathbb{P}_\ell$. 
        Indeed, note that by item \ref{comp} of Proposition \ref{l_struct}, $\widehat{d}$ is the interior of the complement of $\widehat{a}$ so that $\widehat{d}\cup\widehat{\ell}=\widehat{L}$. 
        By item \ref{sym} of Proposition \ref{l_struct}, $d\vee\ell=1$ as desired. 
    \end{proof}
    We once again may conclude the interpretation of a regular locale is regular in all outer models, not simply forcing extensions. 
    \begin{corollary}
        Suppose $M$ is a transitive model of set theory and $M\models L$ is regular locale. 
        Then $\widehat{L}$ is $T_3$. 
    \end{corollary}

    Unlike $T_1,T_2,$ and $T_3$, when interpretations satisfy the axiom $T_4$ is unclear and does not obviously correspond to the localic notion of normality in either direction. 
We say a locale $L$ is \emph{normal} if whenever $a\vee b=1$ then there are $u,v\in L$ such that $u\wedge v=0$ and $a\vee v=1=b\vee u$.

    \begin{question}
        Suppose $L$ is normal. 
        Is $\widehat{L}$ $T_4$ in all forcing extensions?
    \end{question}
    The converse is also unclear. 
    \begin{question}
        Suppose $\widehat{L}$ is $T_4$ in all forcing extensions. Must $L$ be normal?
    \end{question}

    \section{More reflected properties of locales} \label{More_prop_section}
    \subsection{Compactness}
    \begin{definition}
        A locale $L$ is \emph{compact} if whenever $A\subseteq L$ with $\bigvee A=1$, there is a finite $B\subseteq A$ with $\bigvee B=1$. 
    \end{definition}
    We now show that compactness of a locale corresponds precisely with compactness of interpretations. 
    \begin{proposition}
        The following are equivalent: 
        \begin{enumerate}
            \item $L$ is compact; \label{cpt_locale}
            \item In all forcing extensions, every maximal ideal in $L$ is completely prime over the ground model; \label{ideal}
            \item In all forcing extensions, $\widehat{L}$ is compact; \label{cpt_all}
            \item $\widehat{L}$ is compact after forcing with $\operatorname{Coll}(\omega,|L|)$. \label{cpt_force}
        \end{enumerate}
    \end{proposition}
    \begin{proof}
        For $(\ref{cpt_locale})\Rightarrow(\ref{ideal})$, suppose $I\subseteq L$ is a maximal ideal in some forcing extension. 
        Since $L$ remains a distributive lattice, note that $I$ is a prime ideal. 
        If $A\subseteq I$ is in the ground model, by maximality of $I$ either $\bigvee A\in I$ or there is an $a\in I$ such that $a\vee\bigvee A=1$. 
        In the second case, since $L$ is compact in the ground model, there is a finite $F\subseteq A$ such that $a\vee\bigvee F=1$. 
        Since $a\vee\bigvee F\in I$, we conclude that indeed the first case must occur; that is, $\bigvee A\in I$. 

        For $(\ref{ideal})\Rightarrow(\ref{cpt_all})$, suppose $\cU\subseteq L$ satisfies that whenever $F\subseteq\cU$ is finite, $\bigvee F\neq 1$. 
        That is to say, $\cU$ codes a family of basic open subsets of $\widehat{L}$ such that no finite family covers all of $\widehat{L}$. 
        Then the ideal generated by $\cU$ is proper so let $I\supseteq\cU$ be a maximal ideal. 
        Then by $(\ref{ideal})$, $L\setminus I$ is a point of $\widehat{L}$.
        Since $\cU\subseteq I$, $L\setminus I$ is not contained in $\bigcup_{\ell\in\cU}\widehat{\ell}$ so $\cU$ does not code a cover of $\widehat{L}$. 
        
        Note that $(\ref{cpt_all})\Rightarrow(\ref{cpt_force})$ is trivial.
        For $(\ref{cpt_force})\Rightarrow(\ref{cpt_locale})$, suppose $L$ is not compact and fix $A$ such that $\bigvee A=1$ but there is no finite $B\subseteq A$ with $\bigvee B=1$.  
        By item \ref{sym} of Proposition \ref{l_struct} whenever $B\subseteq A$ is finite, $\bigcup_{b\in B}\widehat{b}=\widehat{\bigvee B}\neq \widehat{L}$; that is, $\{\widehat{\ell}\mid \ell\in A\}$ is an open cover with no finite subcover.  
        
    \end{proof}
    \begin{remark}
        The proof of $(\ref{cpt_locale})\Rightarrow(\ref{cpt_force})$ generalizes to show that if $M\models ZFC$ and $M\models$ $L$ is a frame then $\widehat{L}$ is compact (without needing to be in a forcing extension of $M$). 
    \end{remark}
    \begin{remark}
        If $L$ is regular then every completely prime ideal is maximal. 
        In particular, in a compact Hausdorff locale, the completely prime ideals and maximal ideals coincide, yielding another way to see that the interpretation coincides exactly with the space of maximal filters on the ground model compact subsets. 
    \end{remark}

    \subsection{Connectedness}
    There are a couple notions of a locale being connected for locales; of particular note are $L$ being connected and $L$ being $p$-connected. 
    $\widehat{L}$ being connected in all forcing extensions at least implies the stronger $p$-connected. 
    We note by \cite[XIII.4]{FrmLoc} that $p$-connected is a proper strengthening of connectedness. 
    \begin{definition}
        A frame $L$ is $p$-connected (or product connected) if for all frames $M$, every complemented element in $L\oplus M$ is of the form $1\oplus b$ with $b$ complemented in $M$. 
    \end{definition}
    
    We have the following implications between $p$-connectedness and connectedness in forcing extensions. 
    Obtaining connectedness in all forcing extensions from $p$-connectedness seems unlikely, though we do not have a counterexample. 
    \begin{proposition} \label{conn_equiv}
        The following are equivalent for a frame $L$:
        \begin{enumerate}
            \item $L$ is $p$-connected; \label{pconn}
            \item $\widehat{L}$ is connected in some forcing extension in which $(2^{|L|})^V$ is countable; \label{conn_one}
            \item $\widehat{L}$ is connected in all forcing extensions in which $(2^{|L|})^V$ is countable. \label{conn_all}
        \end{enumerate}
    \end{proposition}
    \begin{proof}

        We show $(\ref{pconn})\Leftrightarrow(\ref{conn_all})$ and $(\ref{conn_one})\Rightarrow(\ref{conn_all})$. 
        For $(\ref{conn_all})\Rightarrow(\ref{pconn})$,
        Suppose $L$ is not $p$-connected and fix $a\in L\oplus M$ complemented and not of the form $1\oplus b$. 
        After forcing with $\operatorname{Coll}(\omega,\max(2^{|L|},|M|))$, $\widehat{a}\subseteq\widehat{L}\times\widehat{M}$ is clopen since $a$ is complemented. 
        Since $a$ is not of the form $1\oplus b$ for any $b\in M$, we find that by item \ref{sym} of Proposition \ref{l_struct}, there is a point that is not in $\widehat{a}$ and not in $\widehat{c}$ for $c=\bigvee\{\widehat{L}\times \widehat{b}\mid b\in M,\varnothing=\widehat{a}\cap (\widehat{L}\times \widehat{b})\}$.
        In particular, $\widehat{a}$ is a clopen subset of $\widehat{L}\times \widehat{M}$ not of the form $\widehat{L}\times U$ for any open $U\subseteq\widehat{M}$ so $\widehat{L}$ must be disconnected. 

        For $(\ref{pconn})\Rightarrow(\ref{conn_all})$, suppose $\mathbb{B}$ is a complete Boolean algebra forcing that $2^{|L|}$ is countable. 
        Suppose $\mathbb{B}$ forces that $\widehat{L}$ is disconnected and let $\dot U,\dot W$ be $\mathbb{B}$-names for subsets of $L$ such that 
        \begin{itemize}
            \item for all $\ell\in U$ and $\ell'\in W$, $\ell\wedge\ell'=0$;
            \item for all $x\in\widehat{L}$ there is an $\ell\in x$ such that $\ell\in U$ or $\ell\in W$;
            \item there are $\ell,\ell'>0$ such that $\ell\in U$, $\ell'\in W$. 
        \end{itemize}
        Note this is possible by taking $U$ and $W$ to be all the basic open sets contained in each set witnessing $\widehat{L}$ is disconnected. 
        Now, let $U^*\in L\oplus\mathbb{B}$ be 
        \[\{(\ell,b)\mid b\Vdash \ell\in U\}\]
        and $W^*\in L\oplus\mathbb{B}$ be 
        \[\{(\ell,b)\mid b\Vdash\ell\in W\}.\]
        Note that $U^*\wedge W^*=0$ and that $0<U^*,V^*<1$. 
        Suppose for contradiction that $U^*\vee W^*<1$. 
        Then by item \ref{sym} of Proposition \ref{l_struct}, in some large collapse extension there is an $(x,G)\in\widehat{L\oplus B}\setminus\widehat{(U^*\vee W^*)}=\widehat{L}\times\widehat{B}\setminus\widehat{(U^*\vee W^*)}$. 
        Then $x$ is a point of $\widehat{L}$ not containing any element of $\dot U(G)$ nor $\dot W(G)$. 
        But the assertion that every point of $\widehat{L}$ contains an element of $\dot U(G)$ or $\dot W(G)$ is $\boldsymbol\Pi^1_1$ and therefore absolute between any universes containing $\dot U(G)$, $\dot W(G)$, and in which $(2^{|L|})^V$ is countable. 

        Finally for $(\ref{conn_one})\Rightarrow(\ref{conn_all})$, we note that the assertion that for any $U,W\subseteq L$ satisfying 
        \begin{itemize}
            \item for all $\ell\in U$ and $\ell'\in W$, $\ell\wedge\ell'=0$;
            \item for all $x\in\widehat{L}$ there is an $\ell\in x$ such that $\ell\in U$ or $\ell\in W$;
            \item there are $\ell,\ell'>0$ such that $\ell\in U$, $\ell'\in W$;
        \end{itemize}
        there is an $x\in \widehat{L}$ containing no element of $U$ or $W$ is $\boldsymbol{\Pi}^1_2$ and is equivalent to the assertion that $\widehat{L}$ is disconnected in any model in which $(2^{|L|})^V$ is countable.
        The implication therefore follows from Shoenfield absoluteness. 
        \end{proof}
        \begin{remark}
            There are connected spaces - in fact locally compact connected spaces - whose locale of open sets is not $p$-connected and therefore do not interpret to connected spaces in all forcing extensions; see, for instance, \cite[XIII.4]{FrmLoc}. 
            However, any connected, locally connected locale is $p$-connected; see \cite{prod_conn}.
        \end{remark}
        Absoluteness again allows us to generalize Proposition \ref{conn_equiv} beyond forcing extensions. 
        Being connected does seem to need an additional quantifier, however, so we need the additional hypothesis that $M$ is uncountable to use Shoenfield absoluteness. 
        \begin{corollary}
            Suppose $M$ is a transitive model of set theory and $M\models L$ is a locale which is not $p$-connected and $\mathcal{P}(L)\cap M$ is countable. 
            Then $\widehat{L}$ is disconnected. 
            Conversely, if $M$ is uncountable and transitive, $M\models L$ is a $p$-connected locale, and $\mathcal{P}(L)\cap M$ is countable then $\widehat{L}$ is connected. 
        \end{corollary}

        \subsection{Open maps}
        We now analyze when a map of frames interprets to an open map. Recall that a continuous function $f\colon X\to Y$ is \emph{open} if the image of every open set is open. 
        The localic analogue is the following:
        \begin{definition}
            A frame map $f\colon L\to M$ is \emph{open} if it is a complete Heyting homomorphism; that is, $f$ preserves arbitrary meets and Heyting implication. 
            Equivalently, $f$ has a left adjoint $f_!$ that satisfies the Frobenius identity
            \[f_!(a\wedge f(b))=    f_!(a)\wedge b.\]
        \end{definition}
        Open frame maps are viewed in the localic world as a weakening of the spatial notion that coincides if the codomain is $T_1$. 
        It is unclear whether open maps of locales with a sufficiently separated codomain will always interpret to open maps; we do, however, have one implication. 
        \begin{proposition} \label{open_coll}
            Suppose $f\colon L\to L'$ is a frame map. 
            If $\widehat{f}$ is open after forcing with $\operatorname{Coll}(\omega,|L|+|L'|)$ then $f$ is open. 
        \end{proposition}
        \begin{proof}
        First note that by variations of Proposition \ref{l_struct}, after forcing with $\operatorname{Coll}(\omega,|L|+|L'|)$, $L$ and $L'$ embed faithfully as Heyting subalgebras of $\mathcal{O}(\widehat{L})$ and $\mathcal{O}(\widehat{M})$ such that all meets and joins that exist in the ground model coincide with their ground model versions. 
            Moreover, $\mathcal{O}\widehat{f}\colon \mathcal{O}(\widehat{L'})\to\mathcal{O}(\widehat{L})$ restricts to $f$. 
            If $\widehat{f}$ is open then $\widehat{f}^{-1}$ is a complete Heyting homomorphism so its restriction to $L$ preserves all meets in $L$ and the Heyting implication. 
            That is, $f$ is an open map of frames. 
            \end{proof}
        
        With the additional hypothesis of surjectivity, however, we obtain the following:
        \begin{proposition} \label{open_emb_prop}
            Suppose $f\colon L\to L'$ is a frame map and $L$ is $T_U$. The following are equivalent: 
            \begin{enumerate}
                \item $f$ is open and surjective; \label{open_sur}
                \item in all forcing extensions, $\widehat{f}$ is an open embedding; \label{open_emb}
                
                \item $\widehat{f}$ is an open embedding after forcing with $\operatorname{Coll}(\omega,|L|+|L'|)$. \label{open_collapse}
            \end{enumerate}
        \end{proposition}
        \begin{proof}
            $(\ref{open_collapse})\Rightarrow(\ref{open_sur})$ follows immediately from item \ref{emb} of Proposition \ref{l_struct} and Proposition \ref{open_coll}. 
            For $(\ref{open_sur})\Rightarrow(\ref{open_emb})$, first note that $\widehat{f}$ is an embedding by item \ref{emb} of Proposition \ref{l_struct}. 
            To see that $\widehat{f}$ is open, we show that the image of $\widehat{\ell}$ under $\widehat{f}$ is $\widehat{f_!(\ell)}$. 
            First, since $f$ is surjective, $ff_!(\ell)=\ell$ so $\widehat{f}^{-1}\widehat{f_!(U)}=\widehat{ff_!(\ell)}=\widehat{\ell}$; in particular, $\widehat{f}[\widehat{\ell}]\subseteq \widehat{f_!(\ell)}$. 
            For the reverse inclusion, suppose $x\in \widehat{f_!(U)}$. 
            Let $y=\{a\in L'\mid f_!(a)\in x\}$; we show that $y\in\widehat{\ell}$ and $\widehat{f}(y)=x$. 
            First, $f_!(\ell)\in x$ by hypothesis and $f_!(0)=0\not\in x$. 
            Note that for any $a,b\in L'$, 
            \[\begin{aligned}
                f_!(a)\wedge f_!(b)&=f_!(a\wedge ff_!(b))\\
                &=f_!(a\wedge b),
            \end{aligned}\]
            where we use surjectivity of $f$ again in the second equality. 
            In particular, since $x$ is a filter on $L$, $y$ is a filter on $L'$. 
            Moreover, since $f_!$ preserves arbitrary joins, $y\in \widehat{L'}$. 
            Finally, $\widehat{f}(y)=\{a\in L\mid f_!f(a)\in x\}$; since $f_!f(a)\leq a$, $\widehat{f}(y)\subseteq x$. 
            That is to say, $x\in\overline{\widehat{f}(y)}$
            Since $L$ is $T_U$, $\widehat{L}$ is $T_1$ by Proposition \ref{T1_equiv} so $x=\widehat{f}(y)$. 
            
        \end{proof}
        It seems plausible but tricky that the extra hypothesis of surjectivity can be removed. 
        The tricky point seems to be to replace $f_!$ with a smaller function that now preserves binary meets while still preserving joins. 
        
        In light of Proposition \ref{open_emb_prop}, the next proposition essentially means that interpretation commutes with directed unions. 
        \begin{proposition} \label{union_prop}
            Suppose $D$ is a directed poset and $(L^{i},p^{ij})_{ij\in D,i\leq j}$ is a directed system of locales with each $p^{ij}$ open and surjective when viewed as a map of frames. 
            Let $L$ be the colimit of the $L_i$. 
            Then $\widehat{L}$ is homeomorphic to the colimit of interpreted diagram $(\widehat{L^i},\widehat{p^{ij}})$. 
        \end{proposition}
        \begin{proof}
            We recall by \cite[Section IV.3.1]{FrmLoc} that the colimit of locales may be explicitly given by $\{x\in\prod_iL_i\mid \forall i\leq j (p^{ij}(x_j)=x_i)\}$ with $p^{ij}$ here denoting the map of frames. 
            For each $i,j\in D$ and $\ell\in L_i$, let $x_\ell(j)=p^{jk}p^{ik}_!(\ell)$ for some $k\geq i,j$. 
            The following claim is essentially the same as \cite[Lemma 6.5.1]{FrmLoc} and is tedious but straightforward:
            \begin{claim}
                $x_\ell$ is a well-defined element of $\operatorname{colim}_jL^j$. 
                Moreover, for each $y\in \operatorname{colim}_jL_j$, $y=\bigvee_{j\in D}x_{y(j)}$. 
            \end{claim}
            \begin{proof}
                Concretely, $p^{jk}_!(\ell')=\bigwedge\{\ell''\mid p_{jk}(\ell'')\geq\ell'\}$. 
                Since each $p^{jk}$ is surjective and preserves arbitrary meets, $p^{jk}p^{jk}_!(\ell')=\ell'$ for each $\ell'\in L^j$. 
                Directedness of $D$ then easily gives well-definedness of $x_\ell(j)$ for each $j$. 
                $x_\ell$ being an element of $\operatorname{colim}_jL^j$ also follows easily: if $j\leq j'$ and $k\geq i,j,j'$ then 
                \[p^{jj'}(x_\ell(j'))=p^{jj'}p^{j'k}p^{ik}_!(\ell)=p^{jk}p^{ik}_!(\ell)=x_\ell(j).\]
                Finally, we note that $x_\ell$ is in fact the smallest $z\in\operatorname{colim}_jL_j$ with $z(j)=\ell$. 
                Since joins in $\operatorname{colim}_jL_j$ are computed pointwise, we obtain the moreover part of the claim.
            \end{proof}
            We now define $\varphi\colon\widehat{\operatorname{colim}_iL_i}\to \operatorname{colim}_i\widehat{L_i}$: given $z\in \widehat{\operatorname{colim}_iL_i}$, fix $i$ such that for some $\ell\in L_i$, $x_\ell\in L_i$. 
            Then let $\varphi(z)$ be the image of $\{\ell\in L_i\mid\exists x\in z(x(i)=\ell)\}$. 
            Since each $p^{jk}$ and $p^{jk}_!$ preserve arbitrary joins, we see that indeed $\{\ell\in L_i\mid\exists x\in z(x(i)=\ell)\}$ is a point of $\widehat{L_i}$. 
            
            For continuity of $\varphi$, note that if $\ell\in L_i$ then for $\pi\colon \coprod_{j\in D}\widehat{L^j}=\widehat{\prod_{j\in D}L^j}\to\operatorname{colim}_jL^j$, we have $\pi^{-1}\pi[\widehat{\ell}]=\widehat{x_\ell}$ so that the $\widehat{x_\ell}$ in fact form a basis. 
            An easy further computation shows that $\varphi$ is indeed the inverse to the canonical map from $\operatorname{colim}_i\widehat{L_i}$ to $\widehat{\operatorname{colim}_iL_i}$.
        \end{proof}
        \subsection{Localic groups}
        By a localic group, we mean a group object in the category of locales; that is, a locale $L$ with locale maps $m\colon L\oplus L\to L$, $e\colon 1\to L$, and $\operatorname{inv}\colon L\to L$ satisfying diagramatic versions of the group axioms. 
        Since interpretation preserves finite products and therefore group objects, we may view localic groups as an organized presentation of a group in all forcing extensions. 
        These include all locally compact topological groups but not, for example, $\mathbb{Q}$: after adding a Cohen real $r$, both $r$ and $\sqrt{2}-r$ are in $\widehat{{\Omega}(\mathbb{Q})}$ but their sum is not. 
        
        Allowing for localic groups in our interpretations allows for interesting examples such as the following:
        \begin{proposition} \label{hom}
            For any discrete abelain groups $A,B$, there is a localic group $\underline{\operatorname{hom}}(A,B)$ such that in all forcing extensions, $\widehat{\underline{\operatorname{hom}}(A,B)}$ is homeomorphic to the topological group $\hom(A,B)$ of homomorphisms from $A$ to $B$ in the compact-open topology.

        \end{proposition}
        \begin{proof}
            A free resolution
            \[0\to \bigoplus_I\ZZ\to\bigoplus_J\ZZ \to A\to 0\]
            induces a localic group homomorphism
             \[\bigoplus_J{\Omega}(\ZZ)\to\bigoplus_I{\Omega}(\ZZ).\]
             We take $\underline{\operatorname{hom}}(A,B)$ to be the kernel of this map. 
             Then since interpretation preserves all limits, the interpretation of this group is the kernel of the map of the induced map of topological abelain groups
             \[\prod_J\ZZ\to\prod_I\ZZ,\]
             which is $\hom(A,B)$. 
        \end{proof}

        \begin{remark}
            If $A=\bigoplus_{\omega_1}\mathbb{Z}$ and $B=\mathbb{Z}$ then $\underline{\operatorname{hom}}(A,B)$ is necessarily nonspatial. 
            Indeed, since $\hom(A,B)=\prod_{\omega_1}\ZZ$ with the product topology in all outer models and $\prod_{\omega_1}\ZZ$ does not interpret to the same product in the forcing extension after collapsing $\omega_1$. 
        \end{remark}
        It seems plausible that Proposition \ref{hom} generalizes to the setting of $A$ a locally compact abelian group and $B$ an abelian localic group by similar methods to \cite{fun_spaces}.

\end{document}